\documentclass[preprint,12ptl]{elsarticle}
\usepackage{amsmath,amsfonts}
\usepackage{algorithmic}
\usepackage{algorithm}
\usepackage{array}
\usepackage[caption=false,font=normalsize,labelfont=sf,textfont=sf]{subfig}
\usepackage{textcomp}
\usepackage{stfloats}
\usepackage{url}
\usepackage{verbatim}
\usepackage{graphicx}
\usepackage{amssymb}

\usepackage{latexsym}
\usepackage{amsthm}

\usepackage{float}
\usepackage[hidelinks]{hyperref}

\newcommand{\field}[1]{\mathbb{#1}}
\newcommand{\N}{\field{N}}
\newcommand{\Z}{\field{Z}}

\newcommand{\R}{\field{R}}
\newcommand{\C}{\field{C}}

\DeclareMathOperator{\re}{\mathrm{Re}}

\DeclareMathOperator{\sign}{sign}

\theoremstyle{remark}
\newtheorem{theorem}{Theorem}[section]
\newtheorem{lemma}[theorem]{Lemma}
\newtheorem{corollary}[theorem]{Corollary}

\newtheorem{definition}[theorem]{Definition}
\newtheorem{example}[theorem]{Example}
\newtheorem{remark}[theorem]{Remark}

\usepackage{bm}

\allowdisplaybreaks[1]

\begin{document}

\begin{frontmatter}
  \setcounter{page}{1}
  %\thispagestyle{empty}

% \bigskip \medskip
%\begin{flushleft}
\title{Fourier Decay of Absolutely and H\"older\\Continuous Functions with Infinitely or\\Finitely Many Oscillations}
%\end{flushleft}
%\vspace{\baselineskip}
%\begin{flushleft}
\author[label1,label2]{Juhani Nissil\"a}
%\end{flushleft}
%% Author affiliation
\affiliation[label1]{organization={Intelligent Machines and Systems, Faculty of Technology, University of Oulu},%Department and Organization
            addressline={P.O. Box 4200}, 
            postcode={FI-90014}, 
            city={Oulu},
            country={FINLAND}}
\affiliation[label2]{organization={Applied and Computational Mathematics, Faculty of ITEE, University of Oulu},%Department and Organization
            addressline={P.O. Box 4500}, 
            postcode={FI-90014}, 
            city={Oulu},
            country={FINLAND}}

%\noindent
\begin{abstract}
The main result of this paper is, that if we suppose that a function is absolutely continuous and uniformly H\"older continuous and that its finite difference function does not oscillate infinitely many times on a bounded interval, then the decay rate of its Fourier coefficients can be estimated exactly. This rate of decay predicts the same uniform H\"older continuity but the two other conditions are not necessary. Examples from the literature and by the author show that none of the assumptions can be relaxed without weakening the decay for some functions. The uniform H\"older continuity of chirps and the decay of their Fourier coefficients are studied. The main result is then applied in the estimation of the error of numerical Weyl fractional derivatives calculated using the discrete Fourier transform. The main result is also extended to Fourier transforms.
\end{abstract}
%\vspace{\baselineskip}
%\begin{flushleft}
%\vspace{\baselineskip}
%\textbf{Mathematics Subject Classification}  42A10 $\cdot$ 42A16 $\cdot$ 42A20 $\cdot$ 42A38 $\cdot$ 26A16 $\cdot$ 26A30 $\cdot$ 26A33 $\cdot$ 26A46 %$\cdot$ 33C60 % trigonometric approximation; Fourier coefficients; Convergence and absolute convergence of Fourier and trigonometric series; Fourier transforms; Hölder/Lipschitz spaces; Singular functions, Cantor functions, functions with other special properties; Fractional derivatives and integrals; Absolutely continuous functions,

%Hypergeometric integrals and functions defined by them ($E$, $G$, and $H$ and $I$ functions)
%\end{flushleft}

%%%%%%%%%%%%%%%%%%%%%%%%%%%%%%%%%%%%%%%%%%%%%%%%%%

\end{frontmatter}

 \section{Introduction}\label{sec:1}

It is a well-known result that the decay of Fourier coefficients or the Fourier transform is related to the smoothness of the function. Intuitively, this is caused by the fact that we decompose the function as a sum or integral of infinitely differentiable cosine and sine functions. Thus, the more irregular the function is, the more its Fourier representation requires high frequency components to represent any abrupt changes in it.

A bijective result to predict the same continuity from the Fourier decay alone exists in the case of $L^2$ H\"older continuity and the decay of the function's $L^2$ approximation. The connection between uniform H\"older continuity and Fourier coefficients is a much more subtle problem. After a review of known classical results, it is shown that a bijective result is not possible since absolute continuity and infinite oscillations affect the decay rate but are not necessary conditions. Specifically, I will prove the following Theorem.

 \begin{theorem}\label{Th1}
For some $m \in \N_0$ and $\mu \in (0,1]$, suppose that $f^{(m)}$ is absolutely continuous, $T$-periodic, $f \in C_{m,\,\mu}[0,T]$, and the number of local maxima and minima of $\Delta_h f^{(m)}$ is uniformly bounded for every $0 < h \leq h_0$. Then the Fourier coefficients of $f$ decay like $c_k(f) = O(1/|k|^{1+m+\mu})$ as $|k| \rightarrow \infty$.
 \end{theorem}

The decay rate $O(1/|k|^{1+m+\mu})$ of Fourier coefficients also implies that $f \in C_{m,\,\mu}[0,T]$. This result is probably known, though it is not explicitly written in the textbooks. I will add a proof of this result in Section \ref{sec:4} using difference calculus and the fact that the falling factorial is asymptotically equal to a polynomial with the same exponent (which is also proved in Section \ref{sec:2}). This decay rate also implies the absolute continuity of $f^{(m)}$ in the case $\mu \in (0.5,1)$ but not in the case $\mu \in (0,0.5)$. This part of the problem was studied by Littlewood, Wiener and Wintner as well as Schaeffer in the 1930s and these results are shortly reviewed in Section \ref{sec:3}. Theorem \ref{Th1} is proved in Section \ref{sec:4}.

To prove that the finite oscillations condition is necessary for Theorem \ref{Th1}, we calculate the uniform H\"older continuity of chirps and estimate their Fourier decay in Section \ref{sec:5}. Then we find out that for some chirps their Fourier coefficients decay slower than $O(1/|k|^{1+m+\mu})$ although they are absolutely continuous and in $C_{m,\,\mu}[0,T]$ with $\mu \in (0,1]$. The integral of the Weierstrass function (discussed in Example \ref{ex:Weierstrass}) serves as another example, as it is in $C_{1,\,\mu}[0,T]$ while its Fourier coefficients decay only like $O(1/|k|^{1+\mu})$

One application of the main result is shown when estimating the $L^\infty$ error of numerical Weyl fractional derivatives calculated using the discrete Fourier transform (DFT) in Section \ref{sec:6}. Finally, the main result is extended to Fourier transforms of functions in the Sobolev space $W_{m+1}^1(\R)$ in Section \ref{sec:7} as follows:

 \begin{theorem}\label{Th2}
Suppose that $f \in W_{m+1}^1(\R) \cap C_{m,\,\mu}(\R)$ for some $m \in \N_0$ and $\mu \in (0,1]$, and the number of local maxima and minima of $\Delta_h f^{(m)}$ is uniformly bounded for all $0 < h \leq h_0$. Then the Fourier transform of $f$ decays like $\widehat{f}(\nu) = O(1/|\nu|^{1+m+\mu})$  as $|\nu| \rightarrow \infty$.
 \end{theorem}

%%%%%%%%%%%%%%%%%%%%%%%%%%%%%%%%%%%%%%%%%%%%%%%%%%

\section{Definitions}\label{sec:2}

We will mostly be interested in signals of length $T$ defined on the interval $[0,T]$ or on $[-T/2,T/2]$. The usual Lebesgue spaces of complex-valued functions on the interval are denoted by $L^p(0,T)$ and on the real line by $L^p(\R)$, with $1 \leq p \leq \infty$. The corresponding discrete spaces of infinitely long complex-valued vectors defined over the integers are denoted by $l^p(\Z)$.

 \begin{definition}\label{Def:boundedvariation}
Function $f: [0,T] \rightarrow \C$ is of \textit{bounded variation}, i.e. $f \in BV[0,T]$, if its \textit{total variation} is finite, i.e.
\begin{equation}
V_0^T(f) = \sup_\mathcal{P} \sum_{k=1}^N |f(t_k) - f(t_{k-1})| < \infty,
\end{equation}
where $t_k$, $k = 0,1,\ldots,N$ is a partition of the interval $[0,T]$ and the supremum is taken over all possible partitions $\mathcal{P}$ (of any number $N$ of points) of $[0,T]$.
 \end{definition}

 \begin{definition}\label{Def:absolutecontinuity}
Function $f:$ $[0,T] \rightarrow \C$ is \textit{absolutely continuous}, i.e. $f \in AC[0,T]$, if its derivative $f'$ exists a.e. on $[0,T]$ and is Lebesgue integrable
\begin{equation}
\int_0^t f'(\tau) \, \mathrm{d}\tau = f(t) - f(0), \hspace{1cm} \text{for every } t \in [0,T].
\end{equation}
If $f^{(m)} \in AC[0,T]$ for $m \in \N_0$, then $f$ belongs to \textit{Sobolev space} $W_{m+1}^1[0,T]$.
 \end{definition}
 
%Let us review a couple of basic results related to $BV$ and $AC$ functions.

 \begin{lemma} \label{lemma:BVsum}
If $f,g \in BV[0,T]$ and $\alpha,\beta \in \C$, then $\alpha f + \beta g \in BV[0,T]$ and
\begin{equation}
V_0^T(\alpha f + \beta g) \leq |\alpha|V_0^T(f) + |\beta|V_0^T(g).
\end{equation}
If $0 \leq a < b < c \leq T$, then $V_a^c(f) = V_a^b(f) + V_b^c(f).$
 \end{lemma}
 
 \begin{proof} %%%%%%%%%%%%%
For real-valued functions: \cite[pp. 328--330]{Kolmogorov}.
 \end{proof} %%%%%%%%%%
 
 \begin{lemma} \label{lemma:ACsum}
If $f,g \in AC[0,T]$ and $\alpha,\beta \in \C$, then $\alpha f + \beta g \in AC[0,T]$.
 \end{lemma}
 
 \begin{proof} %%%%%%%%%%%%%
For real-valued functions: \cite[p. 337]{Kolmogorov}.
 \end{proof} %%%%%%%%%%
 
 \begin{lemma} \label{lemma:BV_AC_difference}
Let $f: [0,T] \rightarrow \R$. If $f \in BV[0,T]$, then it can be represented as the difference of two increasing functions. If $f \in AC[0,T]$, then the two functions are also absolutely continuous. In both cases this representation is
\begin{equation}
f(x) = V_0^x(f) - \big( V_0^x(f) - f(x) \big)
\end{equation}
 \end{lemma}
 
 \begin{proof} %%%%%%%%%%%%%
For real-valued functions: \cite[pp. 331 and 337--338]{Kolmogorov}.
 \end{proof} %%%%%%%%%%
 
 \begin{lemma}[Lebesgue decomposition] \label{lemma:Lebesgue_decomposition}
Any real-valued function $f \in BV[0,T]$ can be represented as the sum $f = \phi + \varphi + \chi$, where $\phi \in AC[0,T]$, $\varphi \in BV[0,T]$ and $\varphi$ is also continuous and its derivative is zero a.e. Finally, $\chi$ is a jump function, i.e. it is piecewise constant with at most countably infinite number of steps. It follows that $f' = \phi$ a.e.
 \end{lemma}
 
 \begin{proof} %%%%%%%%%%%%%
\cite[p. 341]{Kolmogorov}.
 \end{proof} %%%%%%%%%%
 
 \begin{definition}\label{Def:holdercontinuity}
Function $f:$ $[0, T] \rightarrow \C$ is \textit{uniformly H\"older continuous} of order $\mu = (0,1]$, if
\begin{equation}\label{eq:holdercontinuity}
|f(t + h) - f(t)| \leq C h^\mu,
\end{equation}
holds for all $t$, $t+h \in [0,T]$ and $0 < h \leq h_0$, where $h_0$ is some sufficiently small number. Then we write $f \in C_\mu[0,T]$. If for some $m = 1, 2, 3, \ldots$ it holds that $f^{(m)} \in C_\mu[0,T]$, we write $f \in C_{m,\,\mu}[0,T]$.
 \end{definition}
 
The case $\mu = 1$ is often called a \textit{Lipschitz condition}. Definitions for pointwise H\"older exponents are discussed for example at the beginning of \cite{HolderExp}. %  Pointwise H\"older conditions are possible to define if one considers for example the equation (\ref{eq:holdercontinuity}) only at the neighbourhood of $t$. 
%and the case $\mu = 0$ would simply mean that $f$ is bounded, but we are mainly interested in the fractional orders of regularity in this paper. The condition (\ref{eq:holdercontinuity}) with $\mu > 1$ is satisfied only by constant functions

 \begin{definition}\label{Def:Lpholdercontinuity}
Let $f \in L^p(0,T)$ and $T$-periodic, where $p \geq 1$. Function $f$ is $L^p$ \textit{H\"older continuous} of order $\mu = (0,1]$, if for all $0 < h \leq h_0$
\begin{equation}
\Vert f_h - f \Vert_p = \left( \frac{1}{T} \int_0^T |f(t+h) - f(t)|^p \mathrm{d}t \right)^{1/p} \leq C h^\mu,
\end{equation}
and we write $f \in C_\mu^p[0,T]$. If $f^{(m)} \in C_\mu^p[0,T]$, we write $f \in C_{m,\,\mu}^p[0,T]$. The case $p = \infty$ would give us just the uniform H\"older continuity.
 \end{definition}
 
%We will also deal with asymptotic notations near the infinities.
The big $O$ and small $o$ notations $O(f)$ and $o(f)$ are used with functions $f: \R \rightarrow \C$ and $O(f_n)$ and $o(f_n)$ with sequences $\{f_n\}_{n=-\infty}^\infty$. Functions $f$ and $g$ (or sequences $f_n$ and $g_n$) are \textit{asymptotically equal} iff
\begin{equation}
\lim_{t\rightarrow \infty}\frac{f(t)}{g(t)} = 1,
\end{equation}
which we notate $f \sim g$ and it implies that $f = O(g)$ and $g = O(f)$. The next Lemma is a classic example of an asymptotic equivalence.

 \begin{lemma}
The Gamma function
\begin{equation}\label{eq:Gamma}
\Gamma(s) = \int_0^\infty x^{s-1}e^{-x} \, \mathrm{d}x, \hspace{0.5cm} \re(s) >0, 
\end{equation}
satisfies the Stirling's formula for $t \in \R_+$
\begin{equation}\label{eq:Stirling}
\Gamma(t+1) \sim \sqrt{2 \pi t} \left( \frac{t}{e} \right)^t, \hspace{0.5cm} \text{as } t \rightarrow \infty.
\end{equation}
 \end{lemma}
 
 \begin{proof} %%%%%%%%%%%%%
Multiple proofs exist in the literature, see for example \cite{Stirlings} or \cite[pp. 12--13]{SpecialFunctions}. The rate of this approximation was first discovered by Abraham de Moivre and the constant was evaluated by Stirling.
 \end{proof} %%%%%%%%%%
 
 \begin{example} \label{ex:asymptotic}
We will also need the following simple asymptotic equivalences, where $t \in \R_+$, $x \in \R$ and $t>x$
\begin{equation}
(t-x)^{t-x} \sim e^{-x}t^{t-x}, \hspace{0.5cm} \text{as } t \rightarrow \infty,
\end{equation}
and
\begin{equation}
\sqrt{t} \sim \sqrt{t-x}, \hspace{0.5cm} \text{as } t \rightarrow \infty.
\end{equation}
The proofs are straightforward calculations
\begin{align}
&\lim_{t\rightarrow\infty} \frac{(t-x)^{t-x}}{t^{t-x}} = \lim_{t\rightarrow\infty} \left( \frac{t-x}{t}\right)^{t-x} = \lim_{t\rightarrow\infty} \left( 1 - \frac{x}{t}\right)^{t-x} \nonumber\\
&= \lim_{t\rightarrow\infty} \exp \left( (t-x) \ln \left( 1 - \frac{x}{t}\right) \right) \nonumber\\
&= \exp \left( \lim_{t\rightarrow\infty} t \ln \left( 1 - \frac{x}{t}\right) \lim_{t\rightarrow\infty} \frac{t-x}{t} \right) \nonumber\\
&= \exp \left( \lim_{t\rightarrow\infty} \frac{\ln \left( 1 - \frac{x}{t}\right)}{1/t}  \lim_{t\rightarrow\infty} \frac{1-\frac{x}{t}}{1} \right) \nonumber\\
&=\exp \left( \lim_{t\rightarrow\infty} \frac{\frac{\mathrm{d}}{\mathrm{d}t} \ln \left( 1 - \frac{x}{t}\right)}{\frac{\mathrm{d}}{\mathrm{d}t} 1/t} \right) \nonumber\\
&= \exp \left( \lim_{t\rightarrow\infty} \frac{ \frac{x}{t^2(1-x/t)} }{-1/t^2} \right) 
= \exp \left( \lim_{t\rightarrow\infty} \frac{x}{\frac{x}{t} -1} \right) = e^{-x},
\end{align}
where we used L'H\^opital's / Johann Bernoulli's rule. This proves the first statement. The second is simpler
\begin{equation}
\lim_{t\rightarrow\infty} \frac{\sqrt{t-x}}{\sqrt{t}} = \lim_{t\rightarrow\infty} \sqrt{1 - \frac{x}{t}} 
= \sqrt{\lim_{t\rightarrow\infty}\left(1 - \frac{x}{t}\right)} = 1.
\end{equation}
 \end{example}

 \begin{definition}\label{Def:fourierseries}
Let $f \in L^1(0,T)$ and $T$-periodic. Its \textit{Fourier series} is
\begin{equation}\label{eq:FourierSeries}
\sum_{k = -\infty}^\infty c_k(f) e^{i 2 \pi k t /T }, 
\end{equation}
where the \textit{Fourier coefficients} $c_k(f)$ are calculated as
\begin{equation}\label{eq:ck}
c_k(f) = \frac{1}{T} \int_{0}^{T} f(t) e^{-i 2 \pi k t/T}\, \mathrm{d}t, \hspace{0.5cm} k = 0, \pm 1,\pm 2,\ldots.
\end{equation}
 \end{definition}
 
  \begin{definition}
Function $f$ belongs to \textit{Sobolev space} $H^s(0,T)$ of order $s > 0$ if
\begin{equation}
\Vert f \Vert_{H^s} = \left(\sum_{k = -\infty}^\infty \left( 1 + |k|^2 \right)^{s} |c_k(f)|^2 \right)^{\frac{1}{2}} < \infty.
\end{equation}
 \end{definition}

 \begin{definition}\label{maar:DFT}
The \textit{Discrete Fourier transform} (DFT) of a sequence \\$\textbf{f} = (f_0, \ldots , f_{N-1}) \in \C^N$ and its \textit{inverse transform} (IDFT) are
\begin{equation}\label{eq:DFT}
\mathcal{F} \{ \textbf{f}\, \}_k = F_k = \frac{1}{N} \sum_{n = 0}^{N-1} f_n e^{-i 2 \pi k n / N}, \hspace{0.5cm} k = 0,\ldots, N-1,
\end{equation}
\begin{equation}\label{eq:IDFT}
\mathcal{F}^{-1} \{ \textbf{F} \}_n = f_n = \sum_{k = 0}^{N-1} F_k e^{i 2 \pi k n / N}, \hspace{0.5cm} n = 0,\ldots, N-1.
\end{equation}
 \end{definition}
 
The IDFT always returns the original sequence \cite[p. 30]{TheDFT}. If we sample a finite length interval with a finer resolution, then the DFT values (\ref{eq:DFT}) approach the Fourier coefficients (\ref{eq:ck}) in the limit $N \rightarrow \infty$ \cite[p. 53]{TheDFT}. In this definition of the DFT, the negative frequencies are located at the points $N/2<k \leq N-1 $. If $N$ is even, the DFT at $k = N/2$ is a combination of the highest resolvable positive and negative frequency.

 \begin{definition}
The \textit{forward difference operator} $\Delta_h$ acts on a function $f$ with a \textit{difference interval} $h >0$
\begin{equation}
\Delta_h f(t) = f(t+h) - f(t) = f_h(t) - f(t).
\end{equation}
 \end{definition}

 \begin{lemma}[The fundamental theorem of sum calculus]\label{lemma:fundamental_theorem_of_sum_calculus}
\begin{equation}
h\sum_{n = 0}^{(N-1)h} f(a+nh) = \left. S(t) \right|_a^{a+Nh},
\end{equation}
where $S$ is such that $\Delta_h S(t) = f(t)$.
 \end{lemma}
 
 \begin{proof} %%%%%%%%%%%%%
\cite[p. 96]{FiniteDifferences}.
 \end{proof} %%%%%%%%%%

 \begin{definition}
The \textit{factorial polynomial} or \textit{falling factorial} is defined as
\begin{equation}
t^{(m)_h} = t(t-h)(t-2h)\dots(t-(m-1)h),
\end{equation}
for $m \in \Z_+$ and some $h >0$ and also $t^{(-m)_h} = 1/\big((t+h)(t+2h)\ldots(t+mh)\big)$. A general definition for all $\gamma  \in \R$ is
\begin{equation}
t^{(\gamma)_h} = \frac{h^\gamma\Gamma\left(\frac{t}{h} + 1\right)}{\Gamma\left(\frac{t}{h} - \gamma +1\right)}.
\end{equation} 
 \end{definition}
 
 \begin{lemma}\label{lemma:factorial_pol_difference}
For all $\gamma \in \R$ the forward difference of a factorial polynomial is another factorial polynomial with the exponent $(\gamma -1)_h$, i.e. for any $h >0$
\begin{equation}
\Delta_h t^{(\gamma)_h} = \gamma h t^{(\gamma-1)_h}.
\end{equation}
 \end{lemma}
 
 \begin{proof} %%%%%%%%%%%%%
\cite[p. 104]{FiniteDifferences}.
 \end{proof} %%%%%%%%%%

Thus the factorial polynomials behave similarly with respect to forward differences as typical polynomials do with respect to differentiation. The three asymptotic equivalences mentioned earlier are utilised to prove the following Lemma.
 
 \begin{lemma}\label{lemma:factorial_pol_asymptotic}
Let $\gamma \in \R$. Then the factorial polynomial is asymptotically equal to a polynomial with the same exponent, i.e.
\begin{equation}
t^{(\gamma)_h} \sim t^\gamma, \hspace{0.5cm} \text{as } t \rightarrow \infty.
\end{equation}
 \end{lemma}
 
 \begin{proof} %%%%%%%%%%%%%
To simplify notations, we first write $t/h = x$. Then we replace the gamma function with its asymptotically equivalent Stirling's formula (\ref{eq:Stirling})

 \begin{equation}
t^{(\gamma)_h} = \frac{h^\gamma\Gamma\left(x + 1\right)}{\Gamma\left(x - \gamma +1\right)} \sim \frac{h^\gamma\sqrt{2 \pi x} \left( \frac{x}{e} \right)^x}{\sqrt{2 \pi x-\gamma} \left( \frac{x-\gamma}{e} \right)^{x-\gamma}}, \hspace{0.5cm} \text{as } t \rightarrow \infty.
 \end{equation}
Next we use the asymptotic formulas from the Example \ref{ex:asymptotic} and tidy up the results 
\begin{equation}
t^{(\gamma)_h} \sim \frac{h^\gamma\sqrt{2 \pi x} \left( \frac{x}{e} \right)^x}{\sqrt{2 \pi x} \left( \frac{x}{e} \right)^{x-\gamma}e^{-\gamma}} = h^\gamma \left( \frac{x}{e} \right)^{\gamma}e^{\gamma} = h^\gamma \left( \frac{t}{h} \right)^\gamma = t^\gamma,
 \end{equation}
as $t \rightarrow \infty$.
 
 \end{proof} %%%%%%%%%%%%%
 
%%%%%%%%%%%%%%%%%%%%%%%%%%%%%%%%%%%%%%%%%%%%%%%%%%
\section{Known results on smoothness and Fourier decay}\label{sec:3}

First we state from the literature that $L^2$ H\"older continuity and the tail sum of Fourier coefficients squared give a simple bijective result.

 \begin{theorem}
Suppose that $f\in L^2(0,T)$ and $T$-periodic. Then \\ $\left(\sum_{|k| > N}|c_k(f)|^2 \right)^{1/2} = O(1/N^{m+\mu})$, for some $m = 0, 1, 2, \ldots$ and $\mu \in (0,1)$ if and only if $f \in C^2_{m,\,\mu}[0,T]$ and $f \in W_{m}^1[0,T]$ (when $m > 0$).
 \end{theorem}
 
 \begin{proof} %%%%%%%%%%%%%
The case $m = 0$ in \cite[p. 40]{Pinsky} and the cases $m > 0$ stated on page 43. The case $m=0$ also in \cite[pp. 41--43]{Serov}.
 \end{proof} %%%%%%%%%%

There are also well-known simple bounds for the decay of Fourier coefficients of H\"older continuous functions.

 \begin{theorem}\label{th:Zygmund_Holder}
Suppose that $f$ is $T$-periodic, $f \in W_{m}^1[0,T]$ (when $m > 0$) and $f \in C_{m,\,\mu}[0,T]$ or $f \in C_{m,\,\mu}^p[0,T]$ for some $\mu \in (0,1]$. Then
\begin{equation}
c_k(f) = O(1/|k|^{m+\mu})
\end{equation}
 \end{theorem}

 \begin{proof} %%%%%%%%%%%%%
Case $m= 0$ proved in \cite[p. 46]{Zygmund1}, \cite[pp. 34--35]{Serov}. For $m > 0$ we use the equality $c_k\left(f^{(m)}\right) = \left(\frac{i2\pi k}{T}\right)^m c_k(f)$, since $f \in W_{m}^1[0,T]$.
 \end{proof} %%%%%%%%%%
 
It is also noteworthy that $\mu$-H\"older continuous functions with $0.5 < \mu < 1$ are "tamer" in many ways when compared to ones with $0 < \mu \leq 0.5$. This is seen for example in the absolute summability of the Fourier coefficients.

 \begin{lemma}
Let $0.5 < \mu <1$ and suppose that $f \in C^2_{\mu}[0,T]$ and $T$-periodic. Then the Fourier coefficients $c_k(f) \in l^1(\Z)$. For $\mu \leq 0.5$ this is not necessarily true.
 \end{lemma}
 \begin{proof} %%%%%%%%%%%%%
\cite[pp. 43--44]{Pinsky}.
 \end{proof} %%%%%%%%%%

In the other direction it is actually easier to study the tail sums of Fourier coefficients, if they are summable, i.e. $c_k(f) \in l^1(\Z)$.
 
 \begin{theorem}\label{theorem:tailsum_holder}
Suppose that $\sum_{|k|> N} |k|^m |c_k(f)| = O(1/N^{\mu})$, for some $m = 0, 1, 2, \ldots$, and $\mu \in (0,1)$. Then $f \in C_{m,\,\mu}[0,T]$.
 \end{theorem}
 
 \begin{proof} %%%%%%%%%%%%%
\cite[pp. 21--22, 46--47]{Serov}.
 \end{proof} %%%%%%%%%%
 
We see that Theorems \ref{th:Zygmund_Holder} and \ref{theorem:tailsum_holder} are not symmetric and that it is generally harder to estimate the uniform H\"older continuity from the Fourier coefficients than to bound the coefficients if the regularity is known. The next two examples show that Theorem \ref{th:Zygmund_Holder} is sharp but also that we can easily find simple H\"older continuous functions whose Fourier coefficients decay faster than it predicts.

\begin{example} \label{ex:Weierstrass}
Many of the first examples of nowhere differentiable but everywhere continuous functions were defined with the help of Fourier series. The first of these published (but not the first discovered) was the \textit{Weierstrass function}
\begin{equation}
\sum_{k=0}^\infty a^{k} \cos(b^k \pi t),
\end{equation}
which for $0 < a < 1$, $ab, b > 1$ is continuous and nowhere differentiable \cite{Hardy}. In Weierstrass's original proof he assumed that $ab > 1+ \frac{3}{2}\pi$ and $b$ is a positive odd integer. Writing $\mu = -\ln(a)/\ln(b)$, we get
\begin{equation}
w_\mu (t) = \sum_{k=1}^\infty b^{-k\mu} \cos(b^k \pi t),
\end{equation}
and from this form it is proved in \cite[p. 47]{Zygmund1} that for $0 < \mu < 1$, $w_\mu \in C_\mu[0,2]$. Writing $m = b^k$, we see that the Fourier coefficients $c_m(w_\mu) = O(|m|^{-\mu})$ \cite[p. 48]{Zygmund1}. The study of these kinds of fractal functions which are defined via the Fourier series is still an active field, as, for example, the paper \cite{HolderExp} demonstrates.
\end{example}

\begin{example} \label{ex:simple_holder_functions}
Let us define a function that is a prototypical example of the sort of functions that satisfy Theorem \ref{Th1}

\begin{equation}
g_\mu(x) = |x|^\mu, \hspace{0.5cm} \text{for } -\pi \leq x \leq \pi,
\end{equation}
where $0 < \mu < 1$. It is stated in \cite[p. 42]{Pinsky} that the Fourier coefficients $c_k(g_\mu) = O(1/|k|^{1+\mu})$, although one can criticise that partial integration was used twice in the proof, since the second application gives a divergent integral. For this reason I will go through the correct proof. The integral is divided into two parts after the partial integration and change of variables $x = t/k$, $\mathrm{d}x = \mathrm{d}t/k$ (here we consider $k > 0$)
\begin{align}
&c_k(g_\mu) = \frac{1}{2\pi} \int_{-\pi}^\pi |x|^\mu e^{-ikx}\,\mathrm{d}x = \frac{1}{\pi} \int_0^\pi x^\mu \cos(kx)\,\mathrm{d}x \nonumber\\
&= 0 - \frac{1}{\pi} \int_0^\pi \mu x^{\mu-1} \frac{\sin(kx)}{k}\,\mathrm{d}x = \frac{-\mu}{\pi k^{1+\mu}} \int_0^{k\pi} t^{\mu-1} \sin(t)\,\mathrm{d}t \nonumber\\
&= \frac{\mu}{\pi k^{1+\mu}} \left( \int_0^\infty t^{\mu-1} \sin(t)\,\mathrm{d}t - \int_{k\pi}^\infty t^{\mu-1} \sin(t)\,\mathrm{d}t \right),
\end{align}
where the first integral is a \textit{Mellin transform} $\mathcal{M}\big\{\sin(t)\big\}(\mu)$ independent of $k$ and the second one a tail integral which decays to zero as $k \rightarrow \infty$, i.e. it is in $o(1)$. This estimation is also discussed in \cite[p. 34]{Serov}.

Figure \ref{fig:alpha07} shows the function $|x|^{0.7}$ sampled on the interval $[-1,1]$ with $2\cdot 10^5$ samples and figure \ref{fig:alpha07_fourier_coeff} shows the decay of its DFT on a log-log scale from $k = 0$ to $10^5-1$. One can estimate the slope of the curve using the marked points
\begin{equation}
\frac{\log(0.003635) - \log(2.073\cdot 10^{-8})}{\log(9) - \log(9999)} = -1.721\ldots \approx -1.7
\end{equation}
so the DFT decays like $O(1/|k|^{1.7})$ to one decimal accuracy. Since the Fourier coefficients are only approximated with the DFT, there is some aliasing error present and it affects the higher frequencies the most.

\begin{figure}[h!]
\centerline{
\includegraphics[scale=0.28]{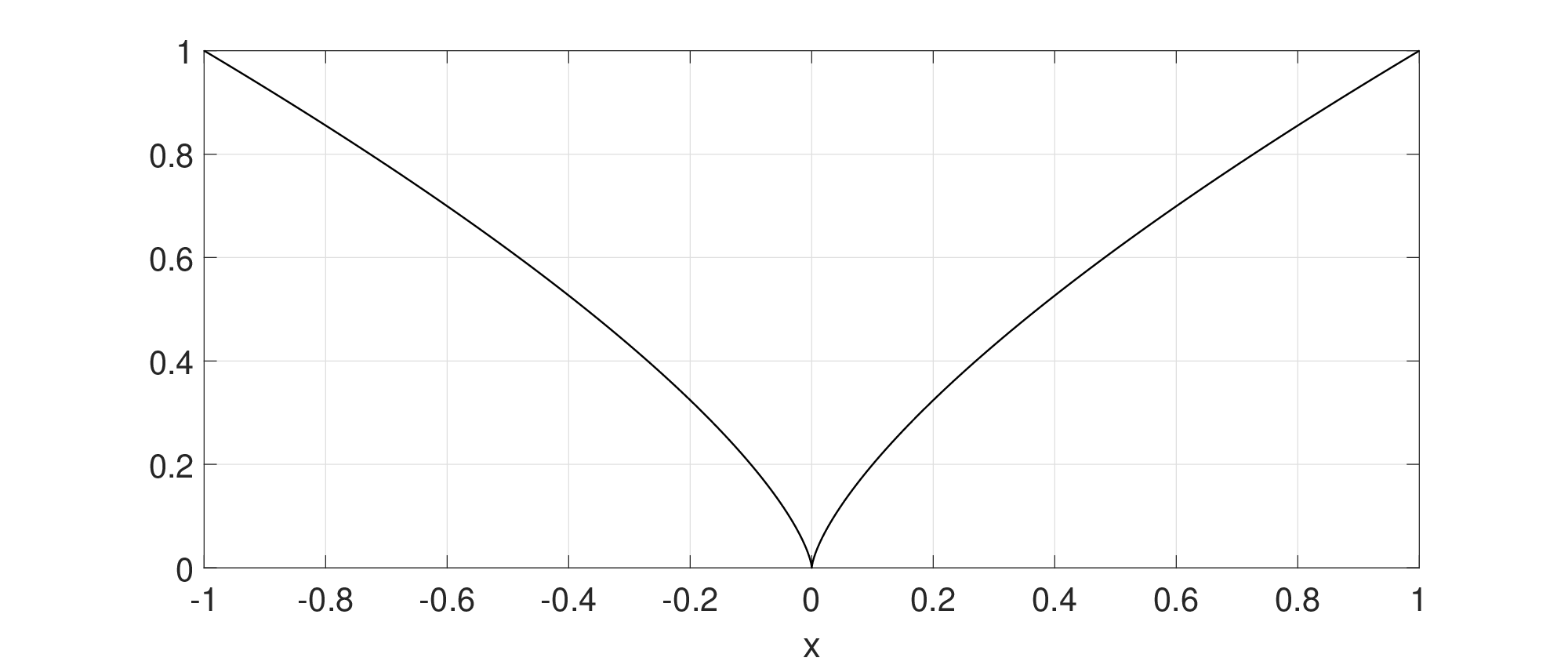}
}
\caption{Function $|x|^{0.7}$ plotted on the interval $[-1,1]$ with $2\cdot 10^5$ samples}
\label{fig:alpha07}
\end{figure}

\begin{figure}[h!]
\centerline{
\includegraphics[scale=0.28]{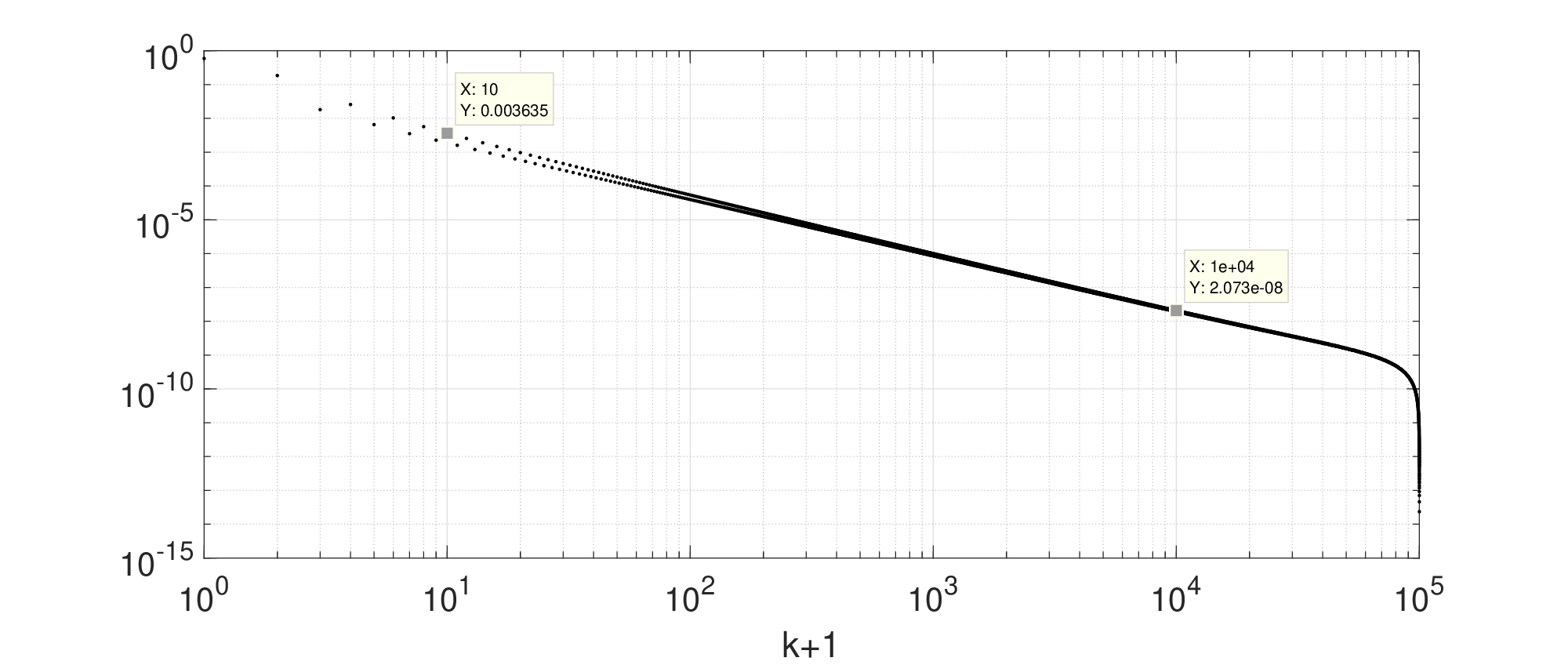}
}
\caption{Absolute values of the DFT of the samples of $|x|^{0.7}$ from figure \ref{fig:alpha07} from $k = 0$ to $10^5-1$ on a log-log scale}
\label{fig:alpha07_fourier_coeff}
\end{figure}
\end{example}

Such examples motivated the author to find exact smoothness conditions which explain the additional $1/|k|$ decay for H\"older continuous functions which behave better than for example the fractal Weierstrass function. Two simple function spaces exist which cause this kind of behaviour.

\begin{lemma}
Suppose that $f \in BV[0,T]$. Then $c_k(f) = O(1/|k|)$.
\end{lemma}
 \begin{proof} %%%%%%%%%%%%%
\cite[p. 48]{Zygmund1}, \cite{Taibleson}.
 \end{proof} %%%%%%%%%%

Similar decay rate is of course achieved for some subsets of $BV[0,T]$, which contains for example piecewise continuous or monotone functions. These kind of special function spaces are also utilised to estimate the errors when estimating Fourier transforms and coefficients with the DFT for example in \cite{Becker_error_FFT}, \cite[pp. 184--211]{TheDFT}.

\begin{lemma}
Suppose that $f \in AC[0,T]$ and $T$-periodic. Then $c_k(f) = o(1/|k|)$.
\end{lemma}
 \begin{proof} %%%%%%%%%%%%%
The Lemma follows directly from the Riemann-Lebesgue-lemma, which states that $c_k(f) = o(1)$ if $f \in L^1(0,T)$. The proof of Riemann-Lebesgue-lemma can be found for example in \cite[p. 45]{Zygmund1}, \cite[pp. 33--34]{Serov} and \cite[p. 18]{Pinsky}.
 \end{proof} %%%%%%%%%%
 
Absolutely continuous functions are of bounded variation \cite[p. 337]{Kolmogorov}, which explains why in the previous Lemmas their Fourier coefficients decay slightly faster. It is also easy to show that Lipschitz continuous functions are in $AC$, but H\"older continuous functions with $0<\mu<1$ are not necessarily so \cite[p. 6]{Fiorenza}. Thus, our investigations to study the bounded variation property and absolute continuity together with H\"older continuity are valid.

Functions which are both H\"older continuous and of bounded variation do indeed have quicker decay of Fourier coefficients in the sense of summability.

\begin{lemma}\label{lemma:BVandHolder}
Suppose that $f \in BV[0,T] \cap C_\mu[0,T]$ with $0 < \mu < 1$ and $f$ is $T$-periodic. Then $c_k(f) \in l^1(\Z)$.
\end{lemma}
 \begin{proof} %%%%%%%%%%%%%
\cite[pp. 57]{Serov}, \cite[p. 44]{Pinsky}.
 \end{proof} %%%%%%%%%%

Nevertheless, we can rule out these functions from our considerations via a counterexample, which is H\"older continuous and of bounded variation but its Fourier coefficients decay only like $c_k(f) = O(1/|k|)$. This is the Cantor-Lebesgue function, another famous function with fractal properties.

\begin{example}\label{ex:Cantor-Lebesgue}
The \textit{Cantor-Lebesgue function} is a continuous and increasing function defined on the interval $[0,1]$. The construction uses the fractal Cantor set, and is described for example in \cite[pp. 334 -- 335]{Kolmogorov} and \cite[pp. 194 -- 196]{Zygmund1}. Zygmund presents the theory for a more general class of functions, Kolmogorov for the classic case where the Cantor set is constructed by always removing the middle thirds of the intervals at each step. The Cantor-Lebesgue function can then be thought of as the cumulative distribution of the Cantor set.

The derivative of this function is 0 almost everywhere although its values increase continuously from 0 to 1. Hence it is not absolutely continuous. It is of bounded variation and the classic case (now denoted by $f$) is also in $C_\mu[0,1]$ with $\mu = \ln(2)/\ln(3)$, which by Lemma \ref{lemma:BVandHolder} means that the Fourier coefficients of the periodic function $f^*(t) = f(t) - t$ are in $l^1(\Z)$. Nevertheless, the decay rate of $c_k(f^*)$ is only $O(1/|k|)$ \cite[pp. 196 -- 197]{Zygmund1}.
\end{example}

Thus, we are left to check the case of functions which are in $AC[0,T] \cap C_\mu[0,T]$. Before proving Theorem \ref{Th1}, let us also state here some results in the other direction from the literature regarding absolute continuity. The first one is a direct consequence of the Riesz-Fischer theorem of the isomorphism between square summable Fourier coefficients and functions in $L^2(0,T)$.

\begin{lemma}\label{lemma:absolute_from_f_coeffs}
Suppose that $f \in H^m(0,T)$ for some $m = 1, 2, 3, \ldots$. Then $f^{(m-1)}$ is a.e. equal to an absolutely continuous function with derivative $f^{(m)}\in L^2(0,T)$ and $c_k\left(f^{(m)}\right) = \left(\frac{i2\pi k}{T} \right)^m \, c_k(f)$.
\end{lemma}

 \begin{proof} %%%%%%%%%%%%%
In \cite[pp. 37 -- 38]{Pinsky} case $m = 1$ is proved and for $m > 1$ stated (follows with induction). Proof of the Riesz-Fischer theorem on page 37 also. Proved also in \cite[pp. 39--40]{Serov}.
 \end{proof} %%%%%%%%%%

The previous result can be easily utilised to obtain an interesting result: suppose that the Fourier coefficients of a $T$-periodic function $f$ decay like $c_k(f) = O(1/|k|^{1+m+\mu})$, with $\mu \in \left(0.5,1\right)$, then $f \in W_{m+1}^1[0,T]$ as well. This follows from the following estimate and Lemma \ref{lemma:absolute_from_f_coeffs}
\begin{equation}\label{eq:simple_functions_AC}
\sum_{k = -\infty}^\infty k^{2(m+1)} |c_k(f)|^2 \leq \sum_{\substack{k = -\infty\\ k \neq 0}}^\infty k^{2m+2} \frac{C}{|k|^{2+2m+2\mu}} = \sum_{\substack{k = -\infty\\ k \neq 0}}^\infty \frac{C}{|k|^{2\mu}} < \infty.
\end{equation}

It is not possible to deduce the absolute continuity from the Fourier coefficients in the case $\mu \in (0, 0.5)$. First hint into this direction was made in 1936 by J.E. Littlewood \cite{Littlewood} with the following counterexample.

 \begin{theorem}\label{th:Littlewood}
There exists an increasing function $f$ with $f'(t) = 0$ for a.e. $t \in [0, T]$ (hence $f$ is not absolutely continuous) and a positive real number $\mu$ such that the Fourier coefficients of the periodic function $f^*(t) = f(t) - \frac{t}{T} \big(f(T)-f(0)\big)$ decay like 
\begin{equation}
c_k(f^*) = O\left(1/|k|^{1+\mu}\right).
\end{equation}
 \end{theorem}
From (\ref{eq:simple_functions_AC}) we know that the number $\mu$ can be at most 1/2 in Littlewood's Theorem. It was actually proved by Wiener and Wintner in 1938 \cite{WienerWintner} that for every $\mu < 1/2$ such non-absolutely continuous functions exist. This proof can also be found in \cite[pp. 146 -- 147]{Zygmund2}. In 1939 \cite{Schaeffer} Schaeffer sharpened this result slightly by proving that for any increasing sequence $r(k)$ that approaches $\infty$ as $k \rightarrow \infty$ (no matter how slowly), there exists a non-absolutely continuous function $f$ whose Fourier coefficients are
\begin{equation}
c_k(f) = O\left(\frac{r(|k|)}{|k|^{1.5}}\right),
\end{equation}
but whether the case $\mu = 1/2$ in Theorem \ref{th:1part2} implies absolute continuity, is probably still an open question.

For completeness, it is fitting to mention here one famous and also current example of these singular non-absolutely continuous functions. To this end, it is helpful to define the Fourier-Stieltjes coefficients and consider their relationship with the Fourier coefficients.

 \begin{definition}
\textit{The Fourier-Stieltjes coefficients} of $g \in BV[0,T]$ are defined by the Riemann-Stieltjes integral
\begin{equation}
c^{\text{FS}}_k(g) = \frac{1}{T}\int_0^T e^{-i2\pi kt/T} \mathrm{d} g(t)
\end{equation}
 \end{definition}
 
 \begin{lemma}\label{FSC_and_FC}
For an increasing function $g\in BV[0,T]$ and its periodic extension $g^*(t) = g(t) - \frac{t}{T} \big(g(T)-g(0)\big)$, we have the following equality of Fourier-Stieltjes and Fourier coefficients for all $k \neq 0$
\begin{equation}
c^{\text{FS}}_k(g) = \frac{i2\pi k}{T}c_k(g^*) = c_k\left( \frac{\mathrm{d}}{\mathrm{d}t}g^* \right).
\end{equation}
 \end{lemma}

 \begin{proof}
\cite{Zygmund1}[p. 41].
 \end{proof} 
 
Lemma \ref{FSC_and_FC} has already appeared in disguise in Example \ref{ex:Cantor-Lebesgue} and Theorem \ref{th:Littlewood} when these singular functions were discussed. Thus we see that with the Fourier-Stieltjes coefficients even the derivatives of singular functions can be meaningfully studied with Fourier methods.

 \begin{example}

Let the continued-fraction representation of the real number $x$ be $[a_0;a_1,a_2,\ldots]$. Then \textit{the Minkowski question mark function} $?$ on $x = [0,1]$ is defined by
\begin{equation}
?(x) = a_0 + 2\sum_{n=1}^\infty \frac{(-1)^{n+1}}{2^{a_1+\ldots+a_n}}.
\end{equation}

The function $?(x)$ is strictly increasing and Salem \cite{Salem} proved that $?(x) \in C_\mu[0,1]$ with
\begin{equation}
\mu = \frac{\ln(2)}{2\ln \frac{\sqrt{5}+1}{2}} = 0.72021\ldots,
\end{equation}
and posed the Salem's problem: do the Fourier-Stieltjes coefficients of $?(x)$ decay to zero?

In 2013, T. Jordan and T. Sahlsten answered the Salem's problem affirmatively in a preprint of the paper \cite{JordanAndSahlsten} by proving, that the Fourier-Stieltjes coefficients of $?(x)$ decay like $c^{\text{FS}}_k(?) = O(|k|^{-\epsilon})$ for some unknown $0 < \epsilon \leq 1/2$. From the considerations above, it follows that the Fourier coefficients of $?(x)-x$ decay like $c_k\big(?(x)-x\big) = O(|k|^{-1-\epsilon})$.

Recently, it was also proved by E. P. Golubeva \cite{Golubeva} and N. V. Gorbatyuk \cite{Gorbatyuk}, that the Fourier-Stieltjes coefficients of the inverse of $?(x)$ defined also on $x = [0,1]$ do not decay to zero, i.e. $c^{\text{FS}}_k(?^{-1}) = O(1)$ and $c^{\text{FS}}_k(?^{-1}) \neq o(1)$. Thus in terms of Fourier decay, the function $?^{-1}$ has the same decay rate as the Cantor-Lebesgue function discussed in Example \ref{ex:Cantor-Lebesgue}, since $c_k\big(?^{-1}(x)-x\big) = O(|k|^{-1})$.

 \end{example}

\section{Proof of the main result}\label{sec:4}

 \begin{theorem} \label{th:1part1}
Suppose that $f \in W_{m+1}^1[0,T] \cap C_{m,\,\mu}[0,T]$ with some $\mu \in (0,1]$, $f$ is $T$-periodic and the number of local extrema of $\Delta_h f^{(m)}$ is uniformly bounded for every $0 < h \leq h_0$. Then the Fourier coefficients of $f$ decay like $c_k(f) = O(1/|k|^{1+m+\mu})$.
 \end{theorem}
 
 \begin{proof} %%%%%%%%%%%%%
Suppose that $f \in C_{\mu}[0,T]$. The cases $m = 1, 2, \ldots$ follow with induction. Since $f$ is also absolutely continuous, we know that its derivative $f'$ exists a.e., is integrable and thus we can study the $L^1$ H\"older continuity of $f'$. Let us denote $g = \Delta_h f = f_h - f$. Then by Lemma \ref{lemma:ACsum}, $g$ is absolutely continuous and by Lemma \ref{lemma:BV_AC_difference}, we can decompose it also as a difference of two increasing absolutely continuous functions $g = g_1 - g_2$, where $g_1(t) = V_0^t(g)$ and $g_2(t) = V_0^t(g) - g(t)$. Then $g_1'(t), g_2'(t) \geq 0$ for almost every $t \in [0,T]$ and
\begin{align}
&T \Vert f'_h - f' \Vert_1 = \int_0^T |f'(t+h) - f'(t)| \mathrm{d}t \nonumber\\
&= \int_0^T |g_1'(t) - g_2'(t)| \mathrm{d}t \nonumber\\
&\leq  \int_0^T g_1'(t) \mathrm{d}t + \int_0^T g_2'(t) \mathrm{d}t \nonumber\\
&= g_1(T) - g_1(0) + g_2(T) - g_2(0) \nonumber\\
&= V_0^T(g) - V_0^0(g) + V_0^T(g) - g(T) - V_0^0(g) + g(0) \nonumber\\
&= 2 V_0^T(g) - g(T) + g(0) \nonumber\\
&= 2 V_0^T(g),
\end{align}
since $ V_0^0(g) = 0$ and $g(T) = g(0)$ since $g$ is clearly also $T$-periodic. Now let us partition the interval $[0,T]$ so that the partition points are the local minima and maxima of $g$. Then on all of the $M$ intervals between these points the function $g$ is either increasing or decreasing and by Lemma \ref{lemma:BVsum}

\begin{align}\label{eq:V0T_bounded}
&V_0^T(g) = \sum_{k=1}^M V_{t_{k-1}}^{t_k}(g) = \sum_{k=1}^M \big| g(t_k) - g(t_{k-1}) \big| \nonumber\\
&= \sum_{k=1}^M \big| f(t_k + h) - f(t_k) - f(t_{k-1} + h) + f(t_{k-1}) \big| \nonumber \\ 
&\leq \sum_{k=1}^M \Big(\big| f(t_k + h) - f(t_k) \big| + \big| f(t_{k-1} + h) - f(t_{k-1}) \big| \Big) \nonumber \\
&\leq 2MC |h|^\mu \leq 2LC |h|^\mu,
\end{align}
where $L$ is the supremum of the number of intervals on which $\Delta_h f$ is either increasing or decreasing over all $h>0$ small enough. Since we assumed that the number of local extrema of the function $\Delta_h f$ is uniformly bounded for every $0<h\leq h_0$, this supremum exists and is finite.

Thus $f' \in C^1_\mu[0,T]$ and it follows from Theorem \ref{th:Zygmund_Holder} that $c_k(f') = O(1/|k|^\mu)$. Since $f$ is absolutely continuous, we know that $c_k(f') = \frac{i2\pi k}{T} \, c_k(f)$ and thus $c_k(f) = O(1/|k|^{1+\mu})$.
 \end{proof} %%%%%%%%%%

 \begin{remark} 
I will also outline another possible way to begin the proof. Here we do not use total variations, but instead the properties of derivatives.
 \end{remark}

 \begin{proof} %%%%%%%%%%
Again let us partition the interval $[0,T]$ so that the partition points are the local minima and maxima of $g = \Delta_h f$. Then the derivative $g'$ has a constant sign in any of these intervals. Thus, we can evaluate the integral over all these intervals and get
\begin{align}
T\Vert f'_h - f' \Vert_1 &= \int_0^T |g'(t)| \, \mathrm{d}t = \sum_{k = 1}^M \int_{t_{k-1}}^{t_{k}} | g'(t)| \, \mathrm{d}t \nonumber\\
&= \sum_{k=1}^M \big| g(t_k) - g(t_{k-1}) \big|,
\end{align}
and the rest of the proof follows the previous proof from (\ref{eq:V0T_bounded}) onwards.
 \end{proof} %%%%%%%%%%
 
 \begin{theorem}\label{th:1part2}
Suppose that the Fourier coefficients of a $T$-periodic function $f$ decay like $c_k(f) = O(1/|k|^{1+m+\mu})$, with some $\mu \in (0,1)$. Then $f \in C_{m,\,\mu}[0,T]$. Also if $\mu \in \left(0.5,1\right)$, then $f \in W_{m+1}^1[0,T]$ as well.
 \end{theorem}

 \begin{proof} %%%%%%%%%%%%%
Suppose that $c_k(f) = O(1/|k|^{1 + m +\mu})$. Let us first estimate the tail sum
\begin{equation}
\sum_{|k|> N} |k|^m |c_k(f)| \leq\sum_{|k|> N} \frac{ C_1 }{|k|^{1+\mu}} \leq 2 \sum_{k = N+1}^\infty \frac{ C_1 }{k^{1+\mu}},
\end{equation}
and the sums are still clearly convergent. Then Lemma \ref{lemma:factorial_pol_asymptotic} allows us to change to factorial polynomials and Lemmas \ref{lemma:factorial_pol_difference} and \ref{lemma:fundamental_theorem_of_sum_calculus} to estimate the sum
\begin{align}
\sum_{|k|> N} |k|^m |c_k(f)| &\leq \sum_{k = N+1}^\infty \frac{ C_2 }{k^{(1+\mu)_h}} \leq \frac{C_3}{(N+1)^{(\mu)_h}} \nonumber\\
&\leq \frac{C_4}{(N+1)^\mu} \leq \frac{C_4}{N^\mu},
\end{align}
and Lemma \ref{lemma:factorial_pol_asymptotic} was used again. Now it follows from Theorem \ref{theorem:tailsum_holder} that $f \in C_{m,\,\mu}[0,T]$.

The absolute continuity in the case $\mu \in (0.5, 1)$ was already proved in (\ref{eq:simple_functions_AC}).
 \end{proof} %%%%%%%%%%

 \begin{remark} 
Perhaps a little simpler proof would utilise that fact that we could also replace the infinite sums with integrals, since they are asymptotically equal in these cases. A similar result for multidimensional Fourier series can be found in \cite[p. 178]{Grafakos}, although in one dimension it only states that the decay rate $c_k(f) = O(1/|k|^{1+m+\mu})$ implies $C_\alpha[0,T]$ for all $\alpha < \mu$.
 \end{remark} 

%%%%%%%%%%%%%%%%%%%%%%%%%%%%%%%%%%%%%%%%%%%%%%%%%%
\section{Chirps}\label{sec:5}

Next, we consider probably the simplest infinitely oscillating class of functions. These are called chirps and they have been studied a lot with wavelet theory \cite{Jaffard, Mallat}. The example is a lengthy one, but it provides us the information that the condition for finite oscillations is necessary in Theorem \ref{th:1part1}. 

We will use the following stationary phase approximation of an oscillatory integral in estimating the Fourier coefficients of chirps.

\begin{lemma}\label{lemma:stationary_phase}
Suppose $\phi$ is real-valued and smooth in $(a,b)$, and $|\phi^{(m)}(x)| \geq 1$ for all $x\in (a,b)$. Then
\begin{equation}
\Big| \int_a^b e^{i \phi(x) \nu} \psi(x)\,\mathrm{d}x \Big| \leq d_m \nu^{-\frac{1}{m}} \Big( |\psi(b)| + \int_a^b |\psi'(x)|\,\mathrm{d}x \Big),
\end{equation}
holds when $m \geq 2$ or when $m =1$ and $\phi'(x)$ is monotonic.
\end{lemma}

\begin{proof} %%%%%%%%%%
Proposition 2 and its Corollary in \cite[pp. 332 -- 334]{Stein}.
\end{proof} %%%%%%%%%%

\begin{example}\label{ex:chirp}
Let $\alpha, \beta > 0$. Then for $x \in [0,L]$ we define
\begin{equation}\label{eq:chirp}
f_{\alpha,\beta}(x) = x^\alpha \sin(1/x^\beta).
\end{equation}
It is immediate that (\ref{eq:chirp}) is pointwise $\alpha$-H\"older in the neighbourhood of 0 according to equation (\ref{eq:holdercontinuity}) if $0 < \alpha < 1$ and infinitely differentiable elsewhere. Decay of the wavelet coefficients of such functions reveal both of the exponents $\alpha$ and $\beta$ and thus wavelet analysis is clearly superior to Fourier analysis in the case of pointwise regularity and oscillation. Interested reader may study for example the excellent books \cite{Jaffard, Mallat}. Nevertheless, we are interested in knowing how the frequencies of these signals decay and what is their uniform H\"older continuity.

In \cite[p. 331]{Kolmogorov} it is stated that $f_{\alpha,\beta}$ is of bounded variation iff $\alpha > \beta$ and we will utilise this result now. Thus, let us suppose that $\alpha > \beta$. Due to the Lebesgue decomposition of a function of bounded variation in Lemma \ref{lemma:Lebesgue_decomposition}, we can deduce then that $\varphi$ and $\chi$ in this decomposition are both zero (since $f_{\alpha,\beta}$ is infinitely differentiable a.e. on $[0,L]$ and it is continuous on $[0,L]$) and thus $f_{\alpha,\beta} \in AC[0,L]$ iff $\alpha > \beta$.

Since Theorem \ref{th:1part1} concerns uniform H\"older continuity, we will next show that the functions in question are $C_{\alpha/(1+\beta)}[0,L]$ for any $L >0$ and for all $\alpha, \beta > 0$ such that $\alpha/(1+\beta) \leq 1$. The last limitation comes from the fact, that for the sake of simplicity we do not consider the H\"older continuity of the derivatives of $f_{\alpha,\beta}$ here. The proof of the case $\alpha = \beta = 1$ can be found in \cite[p. 6]{Fiorenza}, \cite{Zhang-Tang} and also in a discussion at the Mathematics Stack Exchange \cite{MathStackEx}, where username Gaultier sketched a proof that $x \sin(1/x)$  is 0.5-H\"older on $[0,1/2\pi]$. What follows is thus a generalisation of these proofs to more general chirps $f_{\alpha,\beta}$ with $\alpha/(1+\beta) \leq 1$.

Let us first suppose that $x,y \in \left[1/\big(2\pi(n+1)\big)^{1/\beta}, 1/(2\pi n)^{1/\beta}\right]$ with some $n = 1,2,3,\ldots$ and thus $x = 1/(2\pi n + \epsilon)^{1/\beta}$ and  $y = 1/(2\pi n + \delta)^{1/\beta}$ where $0\leq \epsilon, \delta, \leq 2\pi$. Then with the help of the Taylor series for sine and the binomial expansion of Newton we estimate
\begin{align}
&|f_{\alpha,\beta}(x) - f_{\alpha,\beta}(y)| \nonumber\\
&= \left| \frac{1}{(2\pi n + \epsilon)^{\frac{\alpha}{\beta}}}\sin(2\pi n + \epsilon) - \frac{1}{(2\pi n + \delta)^{\frac{\alpha}{\beta}}}\sin(2\pi n + \delta) \right| \nonumber\\
&= \left| \frac{1}{(2\pi n + \epsilon)^{\frac{\alpha}{\beta}}}\sin(\epsilon) - \frac{1}{(2\pi n + \delta)^{\frac{\alpha}{\beta}}}\sin(\delta) \right| \nonumber\\
&= \left| \frac{1}{(2\pi n + \epsilon)^{\frac{\alpha}{\beta}}} \big( \epsilon + O(\epsilon^3)\big) - \frac{1}{(2\pi n + \delta)^{\frac{\alpha}{\beta}}}\big(\delta + O(\delta^3)\big)\right| \nonumber\\
&= \left| \frac{(2\pi n + \delta)^{\frac{\alpha}{\beta}}\big( \epsilon + O(\epsilon^3)\big) - (2\pi n + \epsilon)^{\frac{\alpha}{\beta}}\big(\delta + O(\delta^3)\big)}{(4\pi^2 n^2 + 2\pi n\epsilon + 2\pi n \delta + \epsilon\delta)^{\frac{\alpha}{\beta}}} \right| \nonumber\\
&= \bigg| \frac{ \sum_{k=0}^\infty \binom{\alpha/\beta}{k}(2\pi n)^{\frac{\alpha}{\beta}-k} \delta^k\big( \epsilon + O(\epsilon^3)\big)}{(4\pi^2 n^2 + 2\pi n\epsilon + 2\pi n \delta + \epsilon\delta)^{\frac{\alpha}{\beta}}} \nonumber\\
&\hspace{1.0cm} -\frac{\sum_{k=0}^\infty\binom{\alpha/\beta}{k}(2\pi n)^{\frac{\alpha}{\beta}-k} \epsilon^k \big(\delta + O(\delta^3)\big)}{(4\pi^2 n^2 + 2\pi n\epsilon + 2\pi n \delta + \epsilon\delta)^{\frac{\alpha}{\beta}}} \bigg| \nonumber\\
&\leq \left| \frac{ \sum_{k=0}^\infty \binom{\alpha/\beta}{k}(2\pi n)^{\frac{\alpha}{\beta}-k} \left( \delta^k \epsilon\!-\!\epsilon^k\delta\!+\!O(\delta^k\epsilon^3)\!-\!O(\epsilon^k \delta^3) \right)}{(4\pi^2 n^2)^{\frac{\alpha}{\beta}}} \right| \nonumber\\
&\leq d_0 \left| \frac{ n^{\frac{\alpha}{\beta}} ( \epsilon - \delta)}{n^{2\frac{\alpha}{\beta}}} \right| 
\leq \frac{ d_0 | \epsilon - \delta|}{n^{\frac{\alpha}{\beta}}},
\end{align}
where we estimated the series with its biggest term $k = 0$ multiplied by some constant, as $\epsilon$ and $\delta$ are small. Next, we estimate $|x-y|$ from below
\begin{align}
&|x-y| = \left| \frac{(2\pi n + \delta)^{\frac{1}{\beta}} - (2\pi n + \epsilon)^{\frac{1}{\beta}}}{(2\pi n + \delta)^{\frac{1}{\beta}} (2\pi n + \epsilon)^{\frac{1}{\beta}}} \right| \nonumber\\
&\geq \left| \frac{\sum_{k=0}^\infty \binom{1/\beta}{k}(2\pi n)^{\frac{1}{\beta}-k}\delta^k\!-\! \sum_{k=0}^\infty \binom{1/\beta}{k}(2\pi n)^{\frac{1}{\beta}-k}\epsilon^k}{(4\pi n)^{\frac{2}{\beta}}} \right| \nonumber\\
&= \left| \frac{\sum_{k=0}^\infty \binom{1/\beta}{k}(2\pi n)^{\frac{1}{\beta}-k}\big( \delta^k - \epsilon^k\big)}{(4\pi n)^{\frac{2}{\beta}}} \right| \nonumber\\
&\geq d_1 \left| \frac{n^{\frac{1}{\beta}-1}( \delta - \epsilon)}{n^{\frac{2}{\beta}}} \right| = \frac{d_1 | \delta - \epsilon|}{n^{\frac{1}{\beta}+1}},
\end{align}
where we estimated the series from below by leaving only the term $k = 1$. To combine the two estimates, we need to raise the last inequality to a power $u$ such that
\begin{equation}
\left(\frac{1}{\beta} + 1 \right)u = \frac{\alpha}{\beta},
\end{equation}
from which we solve that $u = \alpha/(1+\beta)$. Thus we have
\begin{equation}
|x-y|^{\frac{\alpha}{1+\beta}} \geq \frac{d_2 | \delta - \epsilon|^{\frac{\alpha}{1+\beta}}}{n^{\frac{\alpha}{\beta}}}
\end{equation}
and finally
\begin{equation}
|f_{\alpha,\beta}(x) - f_{\alpha,\beta}(y)| \leq \frac{d_3 | \delta - \epsilon|^{\frac{\alpha}{1+\beta}}}{n^{\frac{\alpha}{\beta}}} \leq d_4 |x-y|^{\frac{\alpha}{1+\beta}}.
\end{equation}
To extend the result to the bigger interval $[0, (1/2\pi)^{1/\beta}]$ let $y$ be as before, but $x < 1/\big(2\pi(n+1)\big)^{1/\beta}$. Then because of the periodicity of sine and the decay of the function $f_{\alpha,\beta}$ towards $0$, we can find $z \in \left[1/\big(2\pi(n+1)\big)^{1/\beta}, 1/(2\pi n)^{1/\beta}\right]$ such that
\begin{align} \label{eq:chirp_regularity}
|f_{\alpha,\beta}(x) - f_{\alpha,\beta}(y)| &\leq |f_{\alpha,\beta}(z) - f_{\alpha,\beta}(y)| \nonumber\\
&\leq d_4 |z-y|^{\frac{\alpha}{1+\beta}} \nonumber\\
&\leq d_4 |x-y|^{\frac{\alpha}{1+\beta}}.
\end{align}
Since the function $f_{\alpha,\beta}$ does not oscillate at values $x > (2/\pi)^{1/\beta}$ and it is infinitely smooth and bounded there, we can conclude that $f_{\alpha,\beta} \in C_{\alpha/(1+\beta)}[0, L]$ for any $L > 0$.

Next we want to bound the decay rate of the Fourier coefficients $c_k(f_{\alpha,\beta})$. We extend the function periodically as an even function for $x \in [-\pi, \pi]$
\begin{equation}
f_{\alpha,\beta}(x) = |x|^\alpha \sin(1/|x|^\beta) = \frac{1}{2i}|x|^\alpha e^{i|x|^{-\beta}} - \frac{1}{2i}|x|^\alpha e^{-i|x|^{-\beta}}.
\end{equation}
The Fourier coefficients are
\begin{align}
&\frac{1}{2\pi} \int_{-\pi}^\pi f_{\alpha,\beta}(x) e^{-ikx}\,\mathrm{d}x = \frac{1}{\pi} \int_0^\pi f_{\alpha,\beta}(x) \cos(kx) \nonumber \\
&= \frac{1}{i4\pi} \int_0^\pi x^\alpha \big( e^{ix^{-\beta}} - e^{-ix^{-\beta}} \big) \big( e^{ikx} + e^{-ikx} \big) \,\mathrm{d}x ,
\end{align}
so there are four oscillatory integrals with differing signs. Let us consider one of them (a similar example with $-2 < \alpha \leq 0$ and $\beta = 1$ is considered in \cite[pp. 339 -- 341]{Stein})
\begin{equation}\label{eq:stat_phase_ck}
\int_0^\pi x^\alpha e^{i(x^{-\beta} + kx)}\,\mathrm{d}x,
\end{equation}
where the phase is $\phi(x) = x^{-\beta}+kx$. This phase is not exactly like in Lemma \ref{lemma:stationary_phase}, since it depends linearly on the frequency variable $k$. Nevertheless, the dominant contribution according to the stationary phase method comes from near the critical point $x_0$ (which is also a function of $k$)
\begin{equation*}
\phi'(x_0) = -\beta x_0^{-\beta-1} + k = 0 \hspace{0.5cm} \Rightarrow  \hspace{0.5cm} x_0 = \Big(\frac{k}{\beta}\Big)^{\frac{-1}{\beta+1}}
\end{equation*}
and $\phi''$ at this point is
\begin{align}\label{eq:stat_phase_phi_second_derivative}
\phi''(x_0) &= \beta(\beta+1)x_0^{-\beta-2} = \beta(\beta+1) \Big(\frac{k}{\beta}\Big)^{\frac{\beta+2}{\beta+1}} \nonumber\\
&= \beta^{\frac{-1}{\beta+1}}(\beta+1) k^{\frac{\beta+2}{\beta+1}}. 
\end{align}
Divide the integral (\ref{eq:stat_phase_ck}) into three parts
\begin{align}
&\int_0^a x^\alpha e^{i(x^{-\beta} + kx)}\,\mathrm{d}x + \int_a^b x^\alpha e^{i(x^{-\beta} + kx)}\,\mathrm{d}x + \int_b^\pi x^\alpha e^{i(x^{-\beta} + kx)}\,\mathrm{d}x \nonumber\\ &= I_1 + I_2 + I_3, 
\end{align}
so that $x_0 \in (a,b)$.
Near the critical point, i.e. in $(a,b)$ we can choose a suitable constant $d$ that is smaller than the constant in front of $k$ in (\ref{eq:stat_phase_phi_second_derivative})
\begin{equation} \label{eq:stat_phase_phi_bound}
\phi''(x) = \beta(\beta+1)x^{-\beta-2} \geq d k^{\frac{\beta+2}{\beta+1}}.
\end{equation}
When using the idea in Lemma \ref{lemma:stationary_phase}  in the case where the phase $\phi$ is a function of the frequency variable, the frequency variable does not factor out in front as $1/\nu^{1/m}$, but rather we have $1/(\phi^{(m)})^{1/m}$ instead.  Thus with $m = 2$, we get the main term of (\ref{eq:stat_phase_ck}) from $I_2$ while keeping in mind, that  $x_0$ and the bound (\ref{eq:stat_phase_phi_bound}) also depend on $k$
\begin{align}
I_2(k) &= O \Big( k^{\frac{\beta+2}{\beta+1}} \Big)^{-1/2} \cdot O\big(x_0^\alpha \big) \nonumber\\
&= O\Big( k^{-\frac{\beta/2+1}{\beta+1}} \Big) \cdot O\Big( k^{-\frac{\alpha}{\beta+1}} \Big) \nonumber\\
&= O\Big( k^{-\frac{1 + \alpha +\beta/2}{\beta+1}} \Big).
\end{align}

At the boundaries $x = -\pi$ and $x = \pi$ our evenly extended function is Lipschitz continuous, and thus the boundaries contribute a term of size $O(|k|^{-2})$. Thus the stationary phase approximation gives the estimate
\begin{equation} \label{eq:stat_phase_result}
c_k(f_{\alpha,\beta}) = O\left(1/|k|^{(1+\alpha + \frac{\beta}{2})/(\beta + 1)}\right) + O\big(1/|k|^2\big)
\end{equation}
As a particular example in the case $\alpha = 0.7$, $\beta = 0.5$ 
\begin{align}
c_k(f_{0.7,\,0.5}) &= O(1/|k|^2) + O\left(1/|k|^{(1+0.7 +0.25)/(0.5 + 1)}\right)  \nonumber\\
&= O\left(1/|k|^{1.3}\right).
\end{align}
We can say that the Fourier coefficients of the chirp $f_{0.7,\,0.5}$ decay like $O\left(1/|n|^{1.3}\right)$. This is verified by numerical calculations. Figure \ref{fig:infinitely_oscillating_alpha07_beta05} shows the function in question calculated on the interval $[-1,1]$ with $2\cdot 10^5$ samples. The absolute values of its DFT from $k = 0$ to $10^5-1$ are shown in figure \ref{fig:infinitely_oscillating_alpha07_beta05_fourier_coeff}
on a log-log scale. The curve can be bounded by a line with a slope
\begin{equation}
\frac{\log(3.708\cdot 10^{-5}) - \log(1.014\cdot 10^{-6})}{\log(541) - \log(8780)} = -1.291\ldots \approx -1.3,
\end{equation}
so indeed the DFT decays like $O\left(1/|n|^{1.3}\right)$ to one decimal accuracy. We also know from equation (\ref{eq:chirp_regularity}) that $f_{0.7,\,0.5} \in C_{7/15}[0, L]$, where $7/15 = 0.4666\ldots$. Thus the infinite oscillations affect the decay rate of the Fourier coefficients and we do not get $O\left(1/|n|^{1.4666\ldots}\right)$ as would be the case without the infinite oscillations according to Theorem \ref{th:1part1}.

Let us write another example of (\ref{eq:stat_phase_result}) with $\alpha = 0.9$ and $\beta = 0.4$
\begin{align}
c_k(f_{0.9,\,0.4}) &= O(1/|k|^2) + O\left(1/|k|^{(1 + 0.9 + 0.2)/(0.4+1)}\right) \nonumber\\
&= O\left(1/|k|^{1.5}\right).
\end{align}
Now based on equation (\ref{eq:chirp_regularity}) the signal is uniformly H\"older continuous with the exponent $9/14 \approx 0.6429$. Note that because $c_k(f_{0.9,\,0.4}) = O\left(1/|k|^{1.5}\right)$, Theorem \ref{th:1part2} reveals only that the H\"older exponent is at least $0.5$.

\begin{figure}[h!]
\centerline{
\includegraphics[scale = 0.28]{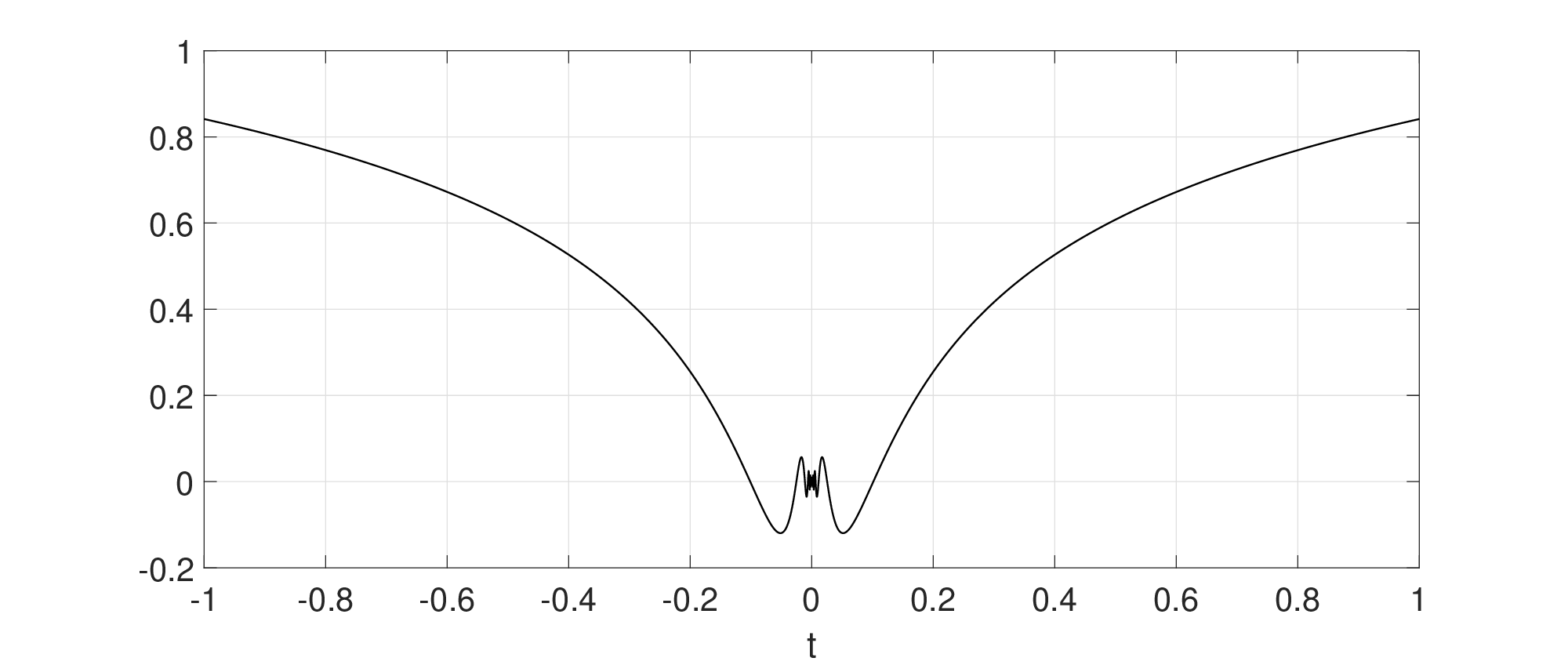}
}
\caption{Function $|x|^{0.7}\sin\left(1/|x|^{0.5}\right)$ on the interval $[-1,1]$ plotted with $2\cdot 10^5$ samples}
\label{fig:infinitely_oscillating_alpha07_beta05}
\end{figure}

\begin{figure}[h!]
\centerline{
\includegraphics[scale = 0.28]{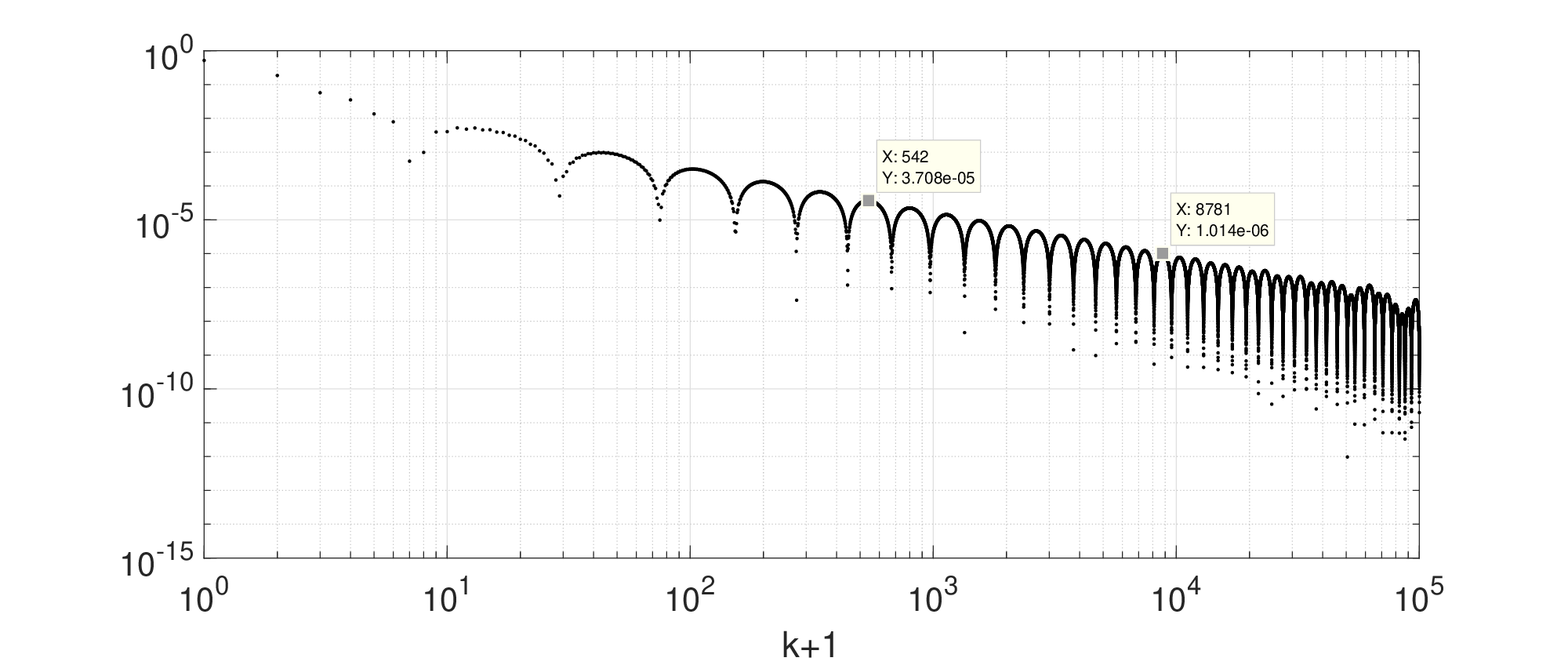}
}
\caption{Absolute values of the DFT of the samples of $|x|^{0.7}\sin\left(1/|x|^{0.5}\right)$ in figure \ref{fig:infinitely_oscillating_alpha07_beta05} from $k = 0$ to $10^5-1$ on a log-log scale}
\label{fig:infinitely_oscillating_alpha07_beta05_fourier_coeff}
\end{figure}

\end{example}

\begin{example}
In this short example we will show that in Theorem \ref{Th1} the $\Delta_h$ operator cannot be removed from the condition: "the number of local maxima and minima of $\Delta_h f^{(m)}$ is uniformly bounded for every $0 < h \leq h_0$". Let us define the chirp

\begin{equation}
f(x) = x^2 \sin(1/|x|) + 3|x|, \hspace{0.5 cm} x \in [-L,L].
\end{equation}
From the previous example we know that $f \in C_1[-L,L]$ and also $c_k(f) = O\left(1/|k|^{(1+2 + \frac{1}{2})/(1 + 1)}\right) = O\left( 1/|k|^{1.75} \right) $. Similarly as the chirps discussed in the previous example, the difference function $\Delta_h f$ has infinitely many local extrema when $h$ is small $0<h<h_0$, and thus this function does not satisfy the requirements of Theorem \ref{Th1}. But we can calculate the derivative of $f$ on $x \in [-L,L]$
\begin{equation}
f'(x) = 2x \sin(1/|x|) -\sign(x) \cdot \cos(1/|x|) + 3\sign(x),
\end{equation}
and note that it is negative on $[-L,0]$ and positive on $[0,L]$. Thus $f$ itself has only finitely many local maxima and minima.

\end{example}

%%%%%%%%%%%%%%%%%%%%%%%%%%%%%%%%%%%%%%%%%%%%%%%%%%
\section{Application to error analysis of numerical Weyl fractional derivatives}\label{sec:6}

In this section we take a quick look at how Theorem \ref{Th1} relates to fractional calculus and its error analysis as derived by the author in \cite{Juhani_error}. The following fractional differintegral operator is a generalisation of the one studied by H. Weyl in \cite{Weyl}.

\begin{definition}\label{maar:Weyl-di} Let $c_k(f)$ be the Fourier coefficients of a $T$-periodic function or distribution $f$. Then the \textit{Weyl differintegral} operator is
\begin{equation}
^{W}D^z f(t) = \sum _{k = -\infty}^\infty \left( \frac{i2\pi k}{T} \right)^{z} c_k(f)  e^{i 2 \pi k t /T },
\end{equation}
where $(i2\pi k/T)^z = e^{z\ln(2\pi|k|/T) + iz\frac{\pi}{2} \text{sign}(k)}$ and $z = \alpha + i \beta$ for some $\alpha, \beta \in \R$.
\end{definition}

The DFT and IDFT pair can be used for differentiation and integration of signals \cite{Juhani_der}. The algorithm consists of calculating the DFT coefficients $F_k$ and then forming a new sequence $\textbf{G} = (G_0, \ldots , G_{N-1})$, with $G_0 = 0$ and
\begin{equation}\label{eq:FourierDInumerical1}
\begin{aligned}
G_k &= \left(i2\pi k/T \right)^{z} F_k, \hspace{1.0cm} 0< k < N/2 \\
G_{N+k} &= \left(i2\pi k/T \right)^{z} F_{N+k}, \hspace{0.5cm} -N/2< k < 0 \\
G_{N/2} &= \left( \frac{\pi N}{T} \right)^{z} \cos\left( z \frac{\pi}{2} \right) F_{N/2} \hspace{0.5cm} \text{(if } N \text{ is even)}.
\end{aligned}
\end{equation}
The estimate for the differintegral $\tilde{\textbf{f}}^{(z)}$ is then calculated with the IDFT
\begin{equation}\label{eq:FourierDInumerical2}
\tilde{f}_n^{(z)} = \mathcal{F}^{-1} \{\textbf{G}\}_n.
\end{equation}
The DFT coefficients (\ref{eq:FourierDInumerical1}) define a trigonometric interpolation also between the discrete points of the estimated fractional derivative as
\begin{equation}\label{eq:trig_interpolant1}
\begin{aligned}
\tilde{f}^{(z)}(t) = &\sum_{0\leq k<N/2} G_k e^{i 2 \pi k t / T} + \sum_{-N/2<k<0} G_{N+k}e^{i 2 \pi k t / T} \\
&+ \left( \frac{\pi N}{T} \right)^{z} F_{N/2} \cos\left(\frac{\pi N t}{T}  + z \frac{\pi}{2}\right),
\end{aligned}
\end{equation}
where the term $F_{N/2}$ is of course included only if $N$ is even.

In \cite{FourDiff1} the following error formula has been proved for $f \in H^s(0,T)$, $s > 1/2$ and $0\leq \sigma \leq s$
\begin{equation}\label{eq:Tadmor_error_1}
\left\Vert f(t) - \tilde{f}(t) \right\Vert_{H^\sigma} \leq \frac{E \Vert f \Vert_{H^s}}{N^{s-\sigma}}.
\end{equation}
although the term $F_{N/2}$ is neqlected in the analysis. A bit more generalised result featuring the same decay rate is also given in eqs. (14) and (15) in \cite{FourDiff3}. An example of the result (\ref{eq:Tadmor_error_1}) for $\sigma = 1$ was also given in \cite{FourDiff1}, in which case for a function $f \in H^s(0,T)$, $s > 1$, the $L^2$ error of the numerical first derivative with the Fourier method is
\begin{equation}\label{eq:Tadmor_error_2}
\left\Vert D^1 f(t) - \tilde{f}^{(1)}(t) \right\Vert_2 \leq \frac{E \Vert f \Vert_{H^s}}{N^{s-1}}.
\end{equation}
This error combines the effects of aliasing and truncation. Similar analysis is done for the Fourier transform and functions in the Sobolev space defined on the real line in \cite{FourDiff2}.

Finally E. Tadmor also stated in \cite{FourDiff1} that a Sobolev inequality implies a result on the maximum norm from (\ref{eq:Tadmor_error_2}) for the first derivative. This Sobolev inequality is the following from $W^q_1(0,T)$ to $L^p(0,T)$ for $1\leq q \leq p \leq \infty$
\begin{equation}
   \left\Vert f(t) \right\Vert_p \leq E' \left\Vert D^1 f(t) \right\Vert_q .
\end{equation}
Explicit calculations of the constant $E'$ for different domains and values of $p$ and $q$ are given in \cite{Mizuguchi2017}. For this analysis, it is enough to know that based on Theorems 2.1 and 3.4 of \cite{Mizuguchi2017} for $p = \infty$ and $q = 2$, we get $E' = O(1/|T|^{1/2}) = O(1/N^{1/2})$. Thus for $f \in H^s(0,T)$, $s > 3/2$
\begin{equation}\label{eq:Tadmor_error_3}
 \left\Vert D^1 f - \tilde{f}^{(1)} \right\Vert_\infty = \max_{t\in [0,T]} \left| D^1 f(t) - \tilde{f}^{(1)}(t) \right| \leq \frac{E'' \Vert f \Vert_{H^s}}{N^{s-3/2}}.
\end{equation}
Note that $D^1 f \in H^{s-1}(0,T)$ and another Sobolev embedding \cite[p. 185]{Serov} guarantees that it is continuous. Since $\tilde{f}^{(1)}$ is also continuous, a maximum exists and no essential supremum is needed in the $L^\infty$ norm.

The author proved the following Theorem about the $L^\infty$ error in \cite{Juhani_error}.

\begin{theorem}\label{th:Juhani_error}
If the Fourier coefficients of the periodic function $f$ decay like $c_k(f) = O(1/|k|^{1+m+\mu})$ with some $m = 0, 1, 2, \ldots$ and $\mu \in (0,1]$, then for $0 < z < m+\mu$
\begin{equation}\label{eq:Juhani_error}
\max_{t\in [0,T]} \left|^{W}D^z f(t) - \tilde{f}^{(z)}(t) \right| \leq \frac{U}{N^{m+\mu-z}},
\end{equation}
where $N$ is the number of discrete samples of the function on the interval $[0,T]$ and $U$ does not depend on $N$. This is a combination of truncation and aliasing errors.
\end{theorem}

Note that also in (\ref{eq:Juhani_error}) the derivative $^{W}D^z f$ is continuous according to Theorem \ref{th:1part2} and also $\tilde{f}^{(z)}$ is continuous. Thus there is a maximum and no essential supremum is needed.

We also note for completeness, that it is shown in \cite{FourDiff1} and \cite{Knockaert} that for analytic functions, the error estimates of the type (\ref{eq:Tadmor_error_1}) decay exponentially for the function and all of its derivatives.

\begin{example}
We can compare the two error results (\ref{eq:Tadmor_error_3}) and (\ref{eq:Juhani_error}) for a periodic function $g$ whose Fourier coefficients are
\begin{equation}
    c_k(g) = \frac{1}{1+|k|^{2+\mu}}
\end{equation}
for $\mu \in (0,1)$ and it follows from \ref{th:1part2} that $g \in C_{1,\mu}[0,T]$.
The condition for $g \in H^s(0,T)$ is
\begin{equation}
\sum_{k = -\infty}^\infty |k|^{2s} |c_k(g)|^2 = \sum_{k = -\infty}^\infty O\left(|k|^{2s-4-2\mu} \right) < \infty,
\end{equation}
from which we get $2s-4-2\mu < -1 \Rightarrow s - 3/2 < \mu$. Since we also needed to have $s - 3/2 > 0$ in order to use (\ref{eq:Tadmor_error_3}), it follows that for any $0 < \epsilon < \mu$ we get that $g \in H^{3/2+\epsilon}(0,T)$ and the following error bound

\begin{align}
\max_{t\in [0,T]}\left|^{W}D^1 f(t) - \tilde{f}^{(1)}(t) \right| &= O\big(1/N^{s-3/2}\big) \\
&= O\big(1/N^{\epsilon}\big),
\hspace{0.5 cm} 0 < \epsilon < \mu.\nonumber
\end{align}

And when using (\ref{eq:Juhani_error}), we don't need the $\epsilon$ to produce the sharper error bound
\begin{equation}
\max_{t\in [0,T]} \left|^{W}D^1 f(t) - \tilde{f}^{(1)}(t) \right| = O\big(1/N^{1+\mu-1}\big) = O\big(1/N^{\mu}\big).
\end{equation}
It is also easier to produce error bounds for arbitrary orders of differentiation using (\ref{eq:Juhani_error}).

\end{example}

%\newpage
%%%%%%%%%%%%%%%%%%%%%%%%%%%%%%%%%%%%%%%%%%%%%%%%%%
\section{Extension to Fourier transforms}\label{sec:7}

Our goal now is to extend Theorem \ref{th:1part1} to non-periodic functions defined on $\R$ and their Fourier transforms. We need to consider integrability and absolute continuity on $\R$.

 \begin{definition}
The \textit{Fourier transform} of a function $f \in L^1(\R)$ is
\begin{equation}\label{eq:FourierL1}
\mathcal{F}\big\{f(t)\big\}(\nu) = \widehat{f}(\nu) = \int_{-\infty}^{\infty} f(t) e^{-i2\pi \nu t} \, \mathrm{d}t
\end{equation}
and its \textit{inverse transform} (if it exists) is
\begin{equation}
\mathcal{F}^{-1}\big\{\widehat{f}(\nu)\big\}(t) = \int_{-\infty}^{\infty} \widehat{f}(\nu) e^{i2\pi \nu t} \, \mathrm{d}\nu.
\end{equation}
 \end{definition}
 
The theory and application of Fourier transforms is simplified considerably if we extend its definition to \textit{tempered distributions} $S'(\R)$. These are linear continuous functionals from the space of \textit{Schwartz functions} $S(\R)$, which are also called \textit{rapidly decaying functions}, to $\C$. As an example of the types of generalised functions in $S'(\R)$, we remind that given any polynomial $q$ and a function $g \in L^p(\R)$ for any $1 \leq p \leq \infty$, their product $pg \in S'(\R)$.

A generalised function $f \in S'(\R)$ has generalised derivatives of all order and they all have Fourier transforms as tempered distributions.

\begin{lemma}\label{lemma:tempered_FT_derivative}
If $f \in S'(\R)$, then $f^{(m)} \in S'(\R)$ for all $m \in \N$. Also $\widehat{f} \in S'(\R)$ and
\begin{equation}
\mathcal{F}\left\{f^{(m)}\right\} = (i2\pi\nu)^m \widehat{f}.
\end{equation}
\end{lemma}

 \begin{proof} %%%%%%%%%%%%%
\cite[pp. 105 -- 106, 185 -- 186]{Zemanian}.
 \end{proof} %%%%%%%%%%%%%

 \begin{definition}\label{Def:boundedvariation_R}
Function $f: \R \rightarrow \C$ is of \textit{bounded variation}, i.e. $f \in BV(\R$), if its \textit{total variation} is finite, i.e. $V_{-\infty}^\infty(f) < \infty$.
 \end{definition}

 \begin{definition}\label{Def:absolutecontinuity_R}
Function $f:$ $\R \rightarrow \C$ is \textit{absolutely continuous}, i.e. $f \in AC(\R)$, if its derivative $f'$ exists a.e. on $\R$ and is Lebesgue integrable
\begin{equation}\label{eq:absolutecontinuity_R}
\int_{-\infty}^t f'(\tau) \, \mathrm{d}\tau = f(t) - C, \hspace{1cm} \text{for every } t \in \R.
\end{equation}

Note that necessarily $C = \lim_{t\rightarrow -\infty} f(t)$ in (\ref{eq:absolutecontinuity_R}). This definition is almost equal to \textit{Sobolev space} $W_1^1(\R)$, which means that the function and its first generalised derivative are in $L_1(\R)$ and that (\ref{eq:absolutecontinuity_R}) applies. Note that the two requirements $f, f' \in L_1(\R)$ of the space $W_1^1(\R)$ guarantee that $\lim_{|t|\rightarrow \infty} f(t) = 0$. We will not need the integrability of $f$ in the Theorem that follows and thus we first state it without Sobolev spaces.
 \end{definition}

The decay rates of Fourier transforms of functions $f \in BV(\R)$ and $f \in AC(\R)$ behave similarly as in the periodic case.

\begin{lemma}\label{lemma:BV_decay_R}
If $f \in BV(\R)$, then $\widehat{f} = O\left(\frac{1}{\nu} \right)$.
\end{lemma}

 \begin{proof} %%%%%%%%%%%%%
Proved using generalised derivatives in \cite[pp. 33 -- 34]{Mallat}.
 \end{proof} %%%%%%%%%%%%%
 
 \begin{lemma}\label{lemma:AC_decay_R}
If $f \in AC(\R)$, then $\widehat{f} = o\left(\frac{1}{\nu} \right)$.
\end{lemma}

 \begin{proof} %%%%%%%%%%%%%
Since $f' \in L^1(\R)$, we know that $f' \in S'(\R)$. A tempered distribution is always a finite order derivative of some continuous function of power growth at the infinities \cite[p. 115]{GelfandShilov2}. Moreover, all the primitives of a tempered distribution are also temperate \cite[p. 108]{Zemanian}, and thus we know that $f$ also has a Fourier transform $\widehat{f} \in S'(\R)$.

Now, according to Lemma \ref{lemma:tempered_FT_derivative}, we get $\widehat{f'} = i2\pi\nu \widehat{f}$. The Riemann-Lebesgue Lemma \cite[p. 94]{Pinsky} also states that $\widehat{f'} = o(1) $, and thus $\widehat{f} = \widehat{f'}/(i2\pi\nu) = o\left(\frac{1}{\nu} \right)$.
 \end{proof} %%%%%%%%%%%%%

 \begin{definition}\label{Def:holdercontinuity_R}
Function $f:$ $\R \rightarrow \C$ is \textit{uniformly H\"older continuous} of order $\mu = (0,1]$, if
\begin{equation}
|f(t + h) - f(t)| \leq C h^\mu,
\end{equation}
holds for all $t$, $t+h \in \R$ and $0 < h \leq h_0$. Then we write $f \in C_\mu(\R)$. If for some $m = 1, 2, 3, \ldots$ it holds that $f^{(m)} \in C_\mu(\R)$, we write $f \in C_{m,\,\mu}(\R)$.
 \end{definition}
 
 \begin{example}\label{ex:AC(R)_example}
Let $\mu \in (0,1]$ . Consider the function
\begin{equation}\label{eq:AC(R)_example_function}
\frac{1}{1+|x|^\mu}.
\end{equation}
Since $|x|^\mu \in C_\mu(\R)$, it follows from the results about H\"older continuity of sums and quotients of H\"older continuous functions in \cite[pp. 14 -- 15]{Fiorenza} that $1/(1+|x|^\mu) \in C_\mu(\R)$. Its derivative is
 \begin{equation}\label{eq:AC(R)_example_function_derivative}
\frac{-\mu \sign(x) |x|^{\mu - 1}}{(1+|x|^\mu)^2} = O(1/|x|^{\mu+1}) \hspace{0.5cm} \text{as } |x|\rightarrow \infty.
\end{equation}
Thus, for every $\mu \in (0,1]$ the function (\ref{eq:AC(R)_example_function}) is not in $L^1(\R)$, but is in $AC(\R)$ because its derivative (\ref{eq:AC(R)_example_function_derivative}) is in $L^1(\R)$ and (\ref{eq:absolutecontinuity_R}) holds.

Note especially that (\ref{eq:AC(R)_example_function_derivative}) is $O(1/|x|^{1-\mu})$ as $|x|\rightarrow 0$, and thus the singularity at 0 is also integrable.
 \end{example} 

 \begin{definition}\label{Def:Lp_R_holdercontinuity}
Function $f \in L^p(\R)$ is $L^p$ H\"older continuous of order $\mu = (0,1]$, if for all $0 < h \leq h_0$
\begin{equation}
\Vert f_h - f \Vert_p = \left( \int_{-\infty}^{\infty} |f(t+h) - f(t)|^p \mathrm{d}t \right)^{1/p} \leq C h^\mu,
\end{equation}
and we write $f \in C_\mu^p(\R)$. If $f^{(m)} \in C_\mu^p(\R)$, we write $f \in C_{m,\,\mu}^p(\R)$.
 \end{definition}

\begin{lemma}\label{lemma:L1_Holder_R}
If $f \in C_{m,\,\mu}^1(\R)$ and $f^{(m-1)} \in AC(\R)$ (when $m > 0$), then $\widehat{f} = O\left(\frac{1}{|\nu|^{m + \mu}} \right)$.
\end{lemma}

 \begin{proof} %%%%%%%%%%%%%
Suppose that $m = 0$. Let us look at
\begin{equation}
\int_{-\infty}^{\infty} f\big(t+1/(2\nu)\big) e^{-i2\pi \nu t} \mathrm{d}t = e^{i\pi} \int_{-\infty}^{\infty} f(\tau) e^{-i2\pi \nu \tau} \mathrm{d}\tau = - \widehat{f}(\nu),
\end{equation}
and thus
\begin{equation}
-2 \widehat{f}(\nu) = \int_{-\infty}^{\infty} \Big( f\big(t+1/(2\nu)\big) - f(t) \Big) e^{-i2\pi \nu t} \mathrm{d}t.
\end{equation}
Now we can bound the Fourier transform
\begin{equation}
\big\vert\widehat{f}(\nu)\big\vert  \leq \frac{1}{2} \int_{-\infty}^{\infty} \Big| f\big(t+1/(2\nu)\big) - f(t) \Big| \mathrm{d}t,
\end{equation}
and if $f \in C_{\mu}^1(\R)$, then
\begin{equation}
\big\vert\widehat{f}(\nu)\big\vert  \leq \frac{C}{2} \left( \frac{1}{2\nu} \right)^\mu.
\end{equation}
If $m > 0$, then we have $\mathcal{F}\left\{f^{(m)}\right\}(\nu) = O(1/|\nu|^\mu)$. Since $f^{(m)} \in L^1(\R)$, we know that $f^{(m)} \in S'(\R)$. A tempered distribution is always a finite order derivative of some continuous function of power growth at the infinities \cite[p. 115]{GelfandShilov2} and all the primitives of a tempered distribution are also temperate \cite[p. 108]{Zemanian}. Thus we know that $f$ also has a Fourier transform $\widehat{f} \in S'(\R)$. Then according the Lemma \ref{lemma:tempered_FT_derivative} we know that $\mathcal{F}\left\{f^{(m)}\right\} = (i2\pi \nu)^m \, \widehat{f}$ and thus $\widehat{f} = O(1/|\nu|^{m+\mu})$.
 \end{proof} %%%%%%%%%%
 
  \begin{remark}
The corresponding Theorem \ref{th:Zygmund_Holder} for Fourier series is valid for all $p \geq 1$ and H\"older's inequality is used to prove this in \cite[pp. 34 -- 35]{Serov}. Using the same inequality here would result in a divergent integral and thus on $\R$ we can only state this result for $p = 1$.
  \end{remark}

 \begin{theorem} \label{th:2part1}
Suppose that $f^{(m)} \in AC(\R)$, $f^{(m)}(t) \rightarrow 0$ as $|t| \rightarrow \infty$ and the number of maxima and minima of $\Delta_h f^{(m)}$ is uniformly bounded for every $0 < h \leq h_0$. Suppose also that $f \in C_{m,\,\mu}[a-h_0,b+h_0]$ with some $\mu \in (0,1]$ and $[a,b]$ contains all the mentioned local extrema of $\Delta_h f^{(m)}$ for all $0 < h \leq h_0$. Then $\widehat{f}(\nu) = O(1/|\nu|^{1+m+\mu})$ as $|\nu| \rightarrow \infty$.
 \end{theorem}
 
 \begin{proof} %%%%%%%%%%
Suppose that $m = 0$. Let us partition $\R$ so that the partition points are the local minima and maxima of $g = \Delta_h f$ and the first point is $-\infty$ and the last point $\infty$. Then the derivative $g'$ has a constant sign in any of these intervals and for any $0 < h \leq h_0$ we have $g(t) \rightarrow 0$ as $|t| \rightarrow \infty$. Thus, we can evaluate
\begin{align}
\Vert f'_h - f' \Vert_1 &= \int_{-\infty}^\infty |g'(t)| \, \mathrm{d}t = \sum_{k = 1}^M \int_{t_{k-1}}^{t_{k}} | g'(t)| \, \mathrm{d}t = \sum_{k=1}^M \big| g(t_k) - g(t_{k-1}) \big| \nonumber\\
&= |g(t_1) - g(-\infty)| + |g(\infty) - g(t_{M-1})| + \sum_{k=2}^{M-1} \big| g(t_k) - g(t_{k-1}) \big|,
\end{align}
and estimate
\begin{align}
\Vert f'_h - f' \Vert_1 &\leq \big| f(t_1 + h) - f(t_1) \big| + \big| f(t_{M-1} + h) - f(t_{M-1}) \big| \nonumber\\
&\enspace\enspace + \sum_{k=2}^{M-1}  \Big(\big| f(t_k + h) - f(t_k) \big| + \big| f(t_{k-1} + h) - f(t_{k-1}) \big| \Big) \nonumber\\
&\leq 2(M-1)C|h|^\mu \leq 2(L-1)C |h|^\mu,
\end{align}
where $L$ is the supremum of the number of increasing and decreasing intervals of $\Delta_h f$ over all $0 < h \leq h_0$.

Thus $f' \in C^1_\mu(\R)$ and because $f \in AC(\R)$ as well, it follows from Lemma \ref{lemma:L1_Holder_R} that $\mathcal{F}\{f\}(\nu) = O(1/|\nu|^{1+\mu})$. The cases $m >0$ are proved identically.
 \end{proof} %%%%%%%%%%
 
 \begin{remark}
If we simplify the conditions of the previous Theorem slightly, a more elegant result follows (this is Theorem \ref{Th2}).
 \end{remark}
 
 \begin{corollary} \label{th:2corollary}
Suppose that $f \in W_{m+1}^1(\R) \cap C_{m,\,\mu}(\R)$ and the number of local extrema of $\Delta_h f^{(m)}$ is uniformly bounded for all $0 < h \leq h_0$. Then $\widehat{f}(\nu) = O(1/|\nu|^{1+m+\mu})$  as $|\nu| \rightarrow \infty$.
 \end{corollary}

 \begin{example}
Based on the observations made in Example \ref{ex:AC(R)_example} it is clear that the function (\ref{eq:AC(R)_example_function}) satisfies the requirements of Theorem \ref{th:2part1} (but not of Corollary \ref{th:2corollary}). It follows that for $\mu \in (0,1]$
\begin{equation}
\mathcal{F}\left\{ \frac{1}{1+|t|^\mu} \right\}(\nu) = O(1/|\nu|^{1+\mu}), \hspace{0.5 cm} \text{as } |\nu| \rightarrow \infty.
\end{equation}

 \end{example}

 \begin{example}

The function $e^{-a|t|}$, $\re(a) >0$ has the Fourier transform
\begin{equation}
\mathcal{F}\left\{ e^{-a|t|} \right\}(\nu) = \frac{2a}{a^2 + (2\pi \nu)^2}.
\end{equation}
Interestingly, this decay rate is explained by two different results. Since the derivative of this function is in $BV(\R)$, Lemma \ref{lemma:BV_decay_R} can be applied. But since the function is Lipschitz at 0, smooth elsewhere and its difference function does not have infinitely many extrema, Corollary \ref{th:2corollary} applies as well.

 \end{example}

\begin{lemma}\label{lemma:convolution_decay}
Let $f_1 \in L^1(\R)$, $f_2 \in L^1(\R)$, and also $f_1 = O(g_1)$ and $f_2 = O(g_2)$ as $|x| \rightarrow \infty$, where $g_1$ and $g_2$ are decreasing functions.
Then the decay of the convolution of $f_1$ and $f_2$ is bounded by
\begin{equation}
    (f_1 \ast f_2) (x) = O(g_1(x)) + O(g_2(x)), \hspace{0.5 cm} \text{as } |x| \rightarrow \infty.
\end{equation}
\end{lemma}

\begin{proof}
This was proved for functions with polynomial decay in \cite{MathStackEx2} and \cite{MathStackEx3}, so we generalise the proof here.
\begin{align}
|(f_1 \ast f_2) (t)| &= \left| \int_{-\infty}^\infty f_1(\tau) f_2(t-\tau) \mathrm{d}\tau \right| \leq \int_{-\infty}^\infty \left| f_1(\tau) f_2(t-\tau) \right| \mathrm{d}\tau \nonumber\\
&= \int_{|\tau|>0.5|t|} \left| f_1(\tau) f_2(t-\tau) \right| \mathrm{d}\tau + \int_{|\tau|<0.5|t|} \left| f_1(\tau) f_2(t-\tau) \right| \mathrm{d}\tau \nonumber\\
&\leq \int_{|\tau|>0.5|t|} C_1 g_1(\tau) |f_2(t-\tau)| \mathrm{d}\tau + \int_{|\tau|<0.5|t|} |f_1(\tau)| C_2 g_2(t-\tau) \mathrm{d}\tau.
\end{align}
And when $|\tau|<0.5|t|$, we have $|t-\tau| \geq |t| - |\tau| \geq |t| - 0.5|t| = 0.5|t|$
\begin{align}
|(f_1 \ast f_2) (t)|
&\leq C_1 g_1(t) \int_{-\infty}^\infty  |f_2(t-\tau)| \mathrm{d}\tau + C_2 g_2(t) \int_{-\infty}^\infty |f_1(\tau)| \mathrm{d}\tau \nonumber\\
&= C_1 g_1(t) \left\Vert f_2 \right\Vert_1 + C_2 g_2(t) \left\Vert f_1 \right\Vert_1, \hspace{0.5 cm}  \text{as } |t| \rightarrow \infty.
\end{align}
\end{proof}

 \begin{example}
For $\mu \in (0,1]$, the function $e^{-\pi t^2}|t|^\mu$ satisfies the requirements of Corollary \ref{th:2corollary}. Thus, we know that $\mathcal{F}\left\{ e^{-\pi t^2}|t|^\mu \right\}(\nu) = O(1/|\nu|^{1+\mu})$ as $|\nu| \rightarrow \infty$. The convolution Theorem for rapidly decaying Schwartz functions $S(\R)$ and tempered distributions \cite[pp. 170 -- 171]{Serov} can be used to decompose this (since $e^{-\pi t^2} \in S(\R)$ and $|t|^\mu \in S'(\R)$ ) when $\mu \in (0,1)$
\begin{align}
\mathcal{F}\left\{ e^{-\pi t^2}|t|^\mu \right\}(\nu) &= \left( \mathcal{F}\left\{ e^{-\pi t^2}\right\}\ast\mathcal{F}\left\{|t|^\mu \right\} \right) (\nu) \nonumber\\
&= e^{-\pi \nu^2} \ast \frac{-2\sin\left( \frac{\pi\mu}{2}\right)\Gamma(\mu +1)}{|2\pi\nu|^{1+\mu}},
\end{align}
where the generalised function $\mathcal{F}\left\{|t|^\mu \right\}(\nu)$ is defined via its Hadamard finite part integral. Therefore, even though the Fourier transforms are known, the latter is not in $L_1(\R)$ and Lemma \ref{lemma:convolution_decay} could not be directly used to prove that the decay rate of $\mathcal{F}\left\{ e^{-\pi t^2}|t|^\mu \right\}(\nu)$ is governed by the decay rate of $\mathcal{F}\left\{|t|^\mu \right\}(\nu)$. The formula for $\mathcal{F}\left\{|t|^\mu \right\}$ and other similar ones is proved, for example, in \cite[pp. 170 -- 173]{GelfandShilov} or in Mark Viola's answer from 2021 in \cite{MathStackEx4}. Since the distributional Fourier transform $\mathcal{F}\left\{|t|^\mu \right\}$ also decays like $O(1/|\nu|^{1+\mu})$, it is likely that a generalisation of Theorem \ref{th:2part1} to some tempered distributions with enough smoothness is possible.

 \end{example}

%newpage

%%%%%%%%%%%%%%%%%%%%%%%%%%%%%%%%%%%%%%%%%%%%%%%%%%

\section*{Acknowledgments}

Most of this work was completed while the author was doing research under a grant from the Finnish Cultural Foundation, North Ostrobothnia Regional Fund. This work has also been supported by the Tauno T\"onning Foundation, Walter Ahlstr\"om Foundation, Auramo Foundation, Otto A. Malm Foundation and Riitta and Jorma J. Takanen Foundation.

I wish to thank Jukka Kemppainen for his comments on an early version of this manuscript. I also thank professor Valery Serov for his excellent courses on the theory of Fourier series and the Fourier transform organised at the University of Oulu.

\begin{figure}[h!]
\centerline{
\includegraphics[scale=1]{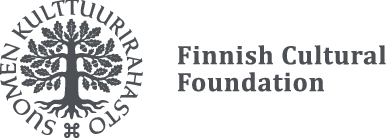}
}
\end{figure}

All web addresses retrieved on 18.03.2026.

  %%%%%%%%%%%%%%%%%%%%%%%%%%%%%%%%

\end{document}